\documentclass[a4paper,10pt,twoside]{article}
 
\usepackage{amsmath,amssymb,amsthm,dsfont}
\usepackage[english]{babel}
\usepackage[bookmarks]{hyperref} 
\usepackage{mathrsfs} 
\usepackage{color,soul} 
\usepackage[titletoc,title]{appendix} 

\usepackage{geometry}
\newtheorem{theorem}{Theorem}[section]

\newtheorem{lem}[theorem]{Lemma} 
\newtheorem{prop}[theorem]{Proposition}
\newtheorem{coro}[theorem]{Corollary} 
\theoremstyle{definition}
\newtheorem{rem}[theorem]{Remark}

\newcommand{\N}{\mathbb N}

\newcommand{\R}{\mathbb R}
\newcommand{\C}{\mathbb C}
\newcommand{\T}{\mathbb T}

\newcommand{\Bc}{\mathcal B}

\newcommand{\Lc}{\mathcal L}

\newcommand{\Rc}{\mathcal R}
\newcommand{\Mc}{\mathcal M}
\newcommand{\Fc}{\mathcal F}

\newcommand{\im}[1]{\mbox{Im} \ #1} 
\newcommand{\scal}[1]{\left\langle #1 \right\rangle} 
\newcommand{\name}{$\underline{\qquad \qquad}$} 

\newcommand{\defendproof}{\hfill $\Box$} 

\usepackage{fancyhdr} 
\begin{document}
\title{\sc Strichartz estimates for the fractional Schr\"odinger and wave equations on compact manifolds without boundary}
\author{ {\sc Van Duong Dinh}} 
\date{ }
\maketitle

\begin{abstract}
We firstly prove Strichartz estimates for the fractional Schr\"odinger equations on $\R^d, d \geq 1$ endowed with a smooth bounded metric $g$. We then prove Strichartz estimates for the fractional Schr\"odinger and wave equations on compact Riemannian manifolds without boundary $(M,g)$. This result extends the well-known Strichartz estimate for the Schr\"odinger equation given in \cite{BGT}. We finally give applications of Strichartz estimates for the local well-posedness of the pure power-type nonlinear fractional Schr\"odinger and wave equations posed on $(M,g)$. 
\end{abstract}




\section{Introduction and main results}
\setcounter{equation}{0}
This paper is concerned with the Strichartz estimates for the generalized fractional Schr\"odinger equation on Riemannian manifold $(M,g)$, namely
\[
\left\{
\begin{array}{ccl}
i\partial_t u + \Lambda_g^\sigma u &=& 0, \\
u(0) &=& u_0, 
\end{array}
\right.
\]
where $\sigma \in (0,\infty)\backslash \{1\}$ and $\Lambda_g=\sqrt{-\Delta_g}$ with $\Delta_g$ is the Laplace-Beltrami operator associated to the metric $g$. When $\sigma \in (0,2)\backslash\{1\}$, it corresponds to the fractional Schr\"odinger equation discorved by N. Laskin (see \cite{Laskin2000, Laskin2002}). When $\sigma \geq 2$, it can be seen as a generalization of the Schr\"odinger equation $\sigma =2$ (see e.g. \cite{Cazenave, Tao}) or the fourth-order Schr\"odinger equation $\sigma =4$ (see e.g. \cite{Pausaderradial, Pausadercubic}).  \newline
\indent The Strichartz estimates play an important role in the study of nonlinear fractional Schr\"odinger equation on $\R^d$ (see e.g. \cite{Cazenave, Tao, Pausaderradial, GuoWang, ChoHwangKwonLee, ChoOzawaXia, HongSire, Dinh} and references therein). Let us recall the local in time Strichartz estimates for the fractional Schr\"odinger operator on $\R^d$. For $\sigma \in (0,\infty)\backslash\{1\}$ and $I \subset \R$ a bounded interval, one has
\begin{align}
\|e^{it\Lambda^\sigma} u_0\|_{L^p(I,L^q(\R^d))} \leq C \|u_0\|_{H^{\gamma_{pq}}(\R^d)}, \label{strichartz estimates on Rd}
\end{align}
where $\Lambda=\sqrt{-\Delta}$ with $\Delta$ is the free Laplace operator on $\R^d$ and 
\[
\gamma_{pq}=\frac{d}{2}-\frac{d}{q}-\frac{\sigma}{p}
\] 
provided that $(p,q)$ satisfies the fractional admissible condition, namely 
\begin{align}
p\in [2,\infty], \quad q \in [2,\infty), \quad (p,q,d) \ne (2,\infty,2), \quad \frac{2}{p}+\frac{d}{q} \leq \frac{d}{2}. \nonumber 
\end{align}
We refer to \cite{Dinh} (see also \cite{ChoOzawaXia}) for a general version of these Strichartz estimates on $\R^d$. \newline
\indent The main purpose of this paper is to prove Strichartz estimates for the fractional Schr\"odinger equation on $\R^d$ equipped with a smooth bounded metric and on a compact manifold without boundary $(M,g)$. \newline
\indent Let us firstly consider $\R^d$ endowed with a smooth Riemannian metric $g$. Let $g(x)=(g_{jk}(x))_{j,k=1}^d$ be a metric on $\R^d$, and denote $G(x)=(g^{jk}(x))_{j,k=1}^d:=g^{-1}(x)$. The Laplace-Beltrami operator associated to $g$ reads
\[
\Delta_g= \sum_{j,k=1}^{d}|g(x)|^{-1} \partial_j \left( g^{jk}(x) |g(x)| \partial_k\right),
\]
where $|g(x)|:=\sqrt{\det g(x)}$ and denote $P:=-\Delta_g$ the self-adjoint realization of $-\Delta_g$. Recall that the principal symbol of $P$ is 
\begin{align}
p(x,\xi)=\xi^t G(x)\xi = \sum_{j,k=1}^d g^{jk}(x)\xi_j \xi_k. \nonumber 
\end{align}
In this paper, we assume that $g$ satisfies the following assumptions.
\begin{enumerate}
\item There exists $C>0$ such that for all $x,\xi\in \R^d$, 
\begin{align}
C^{-1}|\xi|^2 \leq \sum_{j,k=1}^{d}g^{jk}(x)\xi_j\xi_k \leq C |\xi|^2. \label{assump elliptic}
\end{align}
\item For all $\alpha \in \N^d$, there exists $C_\alpha>0$ such that for all $x \in \R^d$,
\begin{align}
| \partial^\alpha g^{jk}(x) | \leq C_\alpha, \quad j,k \in \{1,...d\}. \label{assump bounded metric}
\end{align}
\end{enumerate}
We firstly note that the elliptic assumption $(\ref{assump elliptic})$ implies that $|g(x)|$ is bounded from below and above by positive constants. This shows that the space $L^q(\R^d, d\text{vol}_g), 1 \leq q \leq \infty$ where $d\text{vol}_g= |g(x)| dx$ and the usual Lebesgue space $L^q(\R^d)$ coincide. Thus in the sequel, the notation $L^q(\R^d)$ stands for either $L^q(\R^d,d\text{vol}_g)$ or the usual Lebesgue space $L^q(\R^d)$. It is well-known that under the assumptions $(\ref{assump elliptic})$ and $(\ref{assump bounded metric})$, the Strichartz estimates $(\ref{strichartz estimates on Rd})$ may fail at least for the Schr\"odinger equation (see \cite[Appendix]{BGT}) and in this case (i.e. $\sigma=2$) one has a loss of derivatives $1/p$ that is the right hand side of $(\ref{strichartz estimates on Rd})$ is replaced by $\|u_0\|_{H^{1/p}(\R^d)}$. Here we extend the result of Burq-G\'erard-Tzvetkov to the more general setting, i.e. $\sigma \in (0,\infty) \backslash \{1\}$ and obtain Strichartz estimates with a ``loss'' of derivatives $(\sigma-1)/p$ when $\sigma \in (1,\infty)$ and without ``loss'' when $\sigma \in (0,1)$. Throughout this paper, the ``loss'' compares to $(\ref{strichartz estimates on Rd})$.
\begin{theorem} \label{theorem local strichartz fractional schrodinger on manifolds}
Consider $\R^d, d\geq 1$ equipped with a smooth metric $g$ satisfying $(\ref{assump elliptic}), (\ref{assump bounded metric})$ and let $I \subset \R$ a bounded interval. If $\sigma \in (1,\infty)$, then for all $(p,q)$ fractional admissible, there exists $C>0$ such that for all $u_0 \in H^{\gamma_{pq}+(\sigma-1)/p}(\R^d)$,
\begin{align}
\|e^{it\Lambda_g^\sigma} u_0\|_{L^p(I,L^q(\R^d))} \leq C \|u_0\|_{H^{\gamma_{pq}+(\sigma-1)/p}(\R^d)}, \label{homogeneous strichartz fractional schrodinger on manifolds}
\end{align} 
where $\Lambda_g:=\sqrt{P}$ and $\|u\|_{H^\gamma(\R^d)}:=\|\scal{\Lambda_g}^\gamma u\|_{L^2(\R^d)}$. If $\sigma \in (0,1)$, then $(\ref{homogeneous strichartz fractional schrodinger on manifolds})$ holds with $\gamma_{pq}+(\sigma-1)/p$ is replaced by $\gamma_{pq}$. 
\end{theorem}
The proof of $(\ref{homogeneous strichartz fractional schrodinger on manifolds})$ is based on the WKB approximation which is similar to \cite{BGT}. Since we are working on manifolds, a good way is to decompose the semi-classical fractional Schr\"odinger operator, namely $e^{ith^{-1}(h\Lambda_g)^\sigma}$, in the localized frequency, i.e. $e^{ith^{-1}(h\Lambda_g)^\sigma}\varphi(h^2P)$ for some $\varphi \in C^\infty_0(\R \backslash \{0\})$. The main difficulty is that in general we do not have the exact form of the semi-classical fractional Laplace-Beltrami operator in order to use the usual construction in \cite{BGT}. To overcome this difficulty we write $e^{ith^{-1}(h\Lambda_g)^\sigma}\varphi(h^2P)$ as $e^{-ith^{-1}\psi(h^2P)} \varphi(h^2P)$ where $\psi(\lambda)= \tilde{\varphi}(\lambda) \sqrt{\lambda}^\sigma$ for some $\tilde{\varphi} \in C^\infty(\R \backslash \{0\})$ satisfying $\tilde{\varphi}=1$ near $\text{supp}(\varphi)$. We then approximate $\psi(h^2P)$ in terms of pseudo-differential operators and use the action of pseudo-differential operators on Fourier integral operators in order to construct an approximation for $e^{-ith^{-1}\psi(h^2P)} \varphi(h^2P)$. This approximation gives dispersive estimates for $e^{ith^{-1}(h\Lambda_g)^\sigma}\varphi(h^2P)$ on some small time interval independent of $h$. After scaling in time, we obtain Strichartz estimates without ``loss'' of derivatives over time intervals of size $h^{\sigma-1}$. When $\sigma \in (1,\infty)$, we can cumulate the bounded interval $I$ by intervals of size $h^{\sigma-1}$ and get estimates with $(\sigma-1)/p$ loss of derivatives. In the case $\sigma \in (0,1)$, we can bound the estimates over time intervals of size 1 by the ones of size $h^{\sigma-1}$ and have the same Strichartz estimates as on $\R^d$. It is not a surprise that we recover the same Strichartz estimates as in the free case for $\sigma \in (0,1)$ since $e^{it\Lambda_g^\sigma}$ has micro-locally the finite propagation speed property which is similar to $\sigma =1$ for the (half) wave equation. Intuitively, if we consider the free Hamiltonian $H(x,\xi)=|\xi|^\sigma$, then the spatial component of geodesic flow reads $x(t)=x(0)+t \sigma \xi |\xi|^{\sigma-2}$. After a time $t$, the distance $d(x(t),x(0)) \sim t|\xi|^{\sigma-1} \lesssim t$ if $\sigma -1 \leq 0$ and $|\xi|\geq 1$. By decomposing the solution to $i\partial_t u -\Lambda^\sigma u=0$ as $u = \sum_{k\geq 0} u_k$ where $u_k=\varphi(2^{-k}D) u$ is localized near $|\xi| \sim 2^k \geq 1$, we see that after a time $t$, all components $u_k$ have traveled at a distance $t$ from the data $u_k(0)$. \newline
\indent When $\R^d$ is replaced by a compact Riemannian manifold without boundary $(M,g)$, Burq-G\'erard-Tzvetkov established in \cite{BGT} a Strichartz estimate with loss of $1/p$ derivatives for the Schr\"odinger equation, namely
\begin{align}
\|e^{-it\Delta_g} u_0\|_{L^p(I, L^q(M))} \leq C \|u_0\|_{H^{1/p}(M)}, \label{strichartz estimates schrodinger compact}
\end{align}
where $(p,q)$ is a Schr\"odinger admissible pair, i.e.
\[
p \in [2,\infty], \quad q \in [2,\infty), \quad (p,q,d) \ne (2, \infty,2), \quad \frac{2}{p}+\frac{d}{q} =\frac{d}{2}.
\]
When $M$ is the flat torus $\T^d$, Bourgain showed in \cite{Bourgain}, \cite{Bourgainexponential} some estimates related to $(\ref{strichartz estimates schrodinger compact})$ by means of the Fourier series for the Schr\"odinger equation. A direct consequence of these estimates is
\[
\|e^{-it\Delta_g} u_0\|_{L^4(\T \times \T^d)} \leq C \|u_0\|_{H^\gamma(T^d)}, \quad \gamma >\frac{d}{4}-\frac{1}{2}.
\]
Let us now consider the linear fractional Schr\"odinger equation posed on a compact Riemannian manifold without boundary $(M,g)$, namely
\begin{eqnarray}
\left\{
\begin{array}{ccl}
i\partial_t u(t,x) + \Lambda_g^\sigma u(t,x) &=& F(t,x), \quad (t, x) \in I \times M, \\
u(0,x) &=& u_0(x), \quad x\in M,
\end{array}
\right.
\label{linear fractional schrodinger equation compact no boundary}
\end{eqnarray}
where $\Lambda_g:=\sqrt{-\Delta_g}$ with $\Delta_g$ is the Laplace-Beltrami operator on $(M,g)$. We have the following result.
\begin{theorem} \label{theorem strichartz estimate compact no boundary} Consider $(M,g)$ a smooth compact boundaryless Riemannian manifold of dimension $d \geq 1$ and let $I \subset \R$ a bounded interval. If $\sigma \in (1,\infty)$, then for all $(p,q)$ fractional admissible, there exists $C>0$ such that for all $u_0 \in H^{\gamma_{pq}+(\sigma-1)/p}(M)$, 
\begin{align}
\|e^{it\Lambda_g^\sigma}u_0\|_{L^p(I,L^q(M))} \leq C \|u_0\|_{H^{\gamma_{pq}+(\sigma-1)/p}(M)}. \label{homogeneous strichartz estimate compact no boundary}
\end{align}
Moreover, if $u$ is a (weak) solution to $(\ref{linear fractional schrodinger equation compact no boundary})$, then
\begin{align}
\|u\|_{L^p(I, L^q(M))} \leq C \Big(\|u_0\|_{H^{\gamma_{pq}+(\sigma-1)/p}(M)} + \|F\|_{L^1(I, H^{\gamma_{pq}+(\sigma-1)/p}(M))}\Big). \label{homogeneous strichartz fractional schrodinger with source}
\end{align}
If $\sigma \in (0,1)$, then $(\ref{homogeneous strichartz estimate compact no boundary})$ and $(\ref{homogeneous strichartz fractional schrodinger with source})$ hold with $\gamma_{pq}$ in place of $\gamma_{pq}+(\sigma-1)/p$.
\end{theorem}
\begin{rem} \label{rem homogeneous strichartz fractional schrodinger compact no boundary}
\begin{itemize}
\item[1.] Note that the exponents $\gamma_{pq}+(\sigma-1)/p=d/2-d/q-1/p$ in the right hand side of $(\ref{homogeneous strichartz estimate compact no boundary})$ and $\gamma_{pq}=d/2-d/q-\sigma/p$ in the case of $\sigma \in (0,1)$ correspond to the gain of $1/p$ and $\sigma/p$ derivatives respectively compared with the Sobolev embedding.
\item[2.] When $M=\T$ and $\sigma \in (1,2)$, the authors in \cite{DemirbasErdoganTzirakis} established estimates related to $(\ref{homogeneous strichartz estimate compact no boundary})$, namely
\begin{align}
\|e^{it\Lambda_g^\sigma}u_0\|_{L^4(\T \times \T)} \leq C \|u_0\|_{H^\gamma(\T)}, \quad \gamma >\frac{2-\sigma}{8}. \label{strichartz estimate on torus}
\end{align}
\item[3.] Using the same argument as in \cite{BGT}, we see that the endpoint homogeneous Strichartz estimate $(\ref{homogeneous strichartz estimate compact no boundary})$ are sharp on $\mathbb{S}^d, d\geq 3$. Indeed, let $u_0$ be a zonal spherical harmonic associated to eigenvalue $\lambda = k(d+k-1)$. One has (see e.g. \cite{Soggeoscillatory}) that for $\lambda \gg 1$,
\[
\|u_0\|_{L^q(\mathbb{S}^d)} \sim \sqrt{\lambda}^{s(q)}, \quad s(q)=\frac{d-1}{2}-\frac{d}{q} \quad \text{if } \frac{2(d+1)}{d-1} \leq q \leq \infty.
\]
Moreover, the above estimates are sharp. Therefore,
\[
\|e^{it\Lambda_g^\sigma} u_0\|_{L^2(I, L^{2^\star}(\mathbb{S}^d))} = \|e^{it\sqrt{\lambda}^\sigma} u_0\|_{L^2(I, L^{2^\star}(\mathbb{S}^d))} \sim \sqrt{\lambda}^{s(2^\star)},
\]
where $2^\star=2d/(d-2)$ and $s(2^\star)=1/2$. This gives the optimality of $(\ref{homogeneous strichartz estimate compact no boundary})$ since $\gamma_{22^\star}+(\sigma-1)/2=1/2$. 
\end{itemize}
\end{rem}
A first application of Theorem $\ref{theorem strichartz estimate compact no boundary}$ is the Strichartz estimates for the fractional wave equation posed on $(M,g)$. Let us consider the following linear fractional wave equation posed on $(M,g)$,
\begin{eqnarray}
\left\{
\begin{array}{ccl}
\partial^2_t v(t,x) + \Lambda_g^{2\sigma} v(t,x) &=& G(t,x), \quad (t, x) \in I \times M, \\
v(0,x) = v_0(x),& & \partial_tv(0,x)= v_1(x), \quad x\in M.
\end{array}
\right.
\label{linear fractional wave equation compact no boundary}
\end{eqnarray}
We refer to \cite{ChenHolm} or \cite{Herrmann} for the introduction of fractional wave equations which arise in Physics. We also refer to \cite{Dinh} for Strichartz estimates of the linear fractional wave equations on $\R^d$. As for the linear fractional Schr\"odinger equations, we obtain estimates with a loss of derivatives $(\sigma-1)/p$ when $\sigma\in (1,\infty)$ and with no loss when $\sigma \in (0,1)$. Precisely, we have the following result. 
\begin{coro} \label{coro strichartz estimates wave compact no boundary}
Consider $(M,g)$ a smooth compact boundaryless Riemannian manifold of dimension $d \geq 1$. Let $I \subset \R$ be a bounded interval and $v$ a (weak) solution to $(\ref{linear fractional wave equation compact no boundary})$. If $\sigma \in (1,\infty)$, then for all $(p,q)$ fractional admissible, there exists $C>0$ such that for all $(v_0,v_1) \in H^{\gamma_{pq}+(\sigma-1)/p}(M) \times H^{\gamma_{pq}+(\sigma-1)/p-\sigma}(M)$,
\begin{align}
\|v\|_{L^p(I, L^q(M))} \leq C \Big(\|[v](0)\|_{H^{\gamma_{pq}+(\sigma-1)/p}(M)}  +\|G\|_{L^1(I, H^{\gamma_{pq}+(\sigma-1)/p-\sigma}(M))}\Big), \label{homogeneous strichartz fractional wave with source}
\end{align}
where 
\[
\|[v](0)\|_{H^{\gamma_{pq}+(\sigma-1)/p}(M)}:=\|v_0\|_{H^{\gamma_{pq}+(\sigma-1)/p}(M)}+\|v_1\|_{H^{\gamma_{pq}+(\sigma-1)/p-\sigma}(M)}.
\]
If $\sigma \in (0,1)$, then $(\ref{homogeneous strichartz fractional wave with source})$ holds with $\gamma_{pq}+(\sigma-1)/p$ is replaced by $\gamma_{pq}$.
\end{coro}
\begin{rem} \label{rem strichartz fractional wave bounded metric}
A similar result (with the same proof) as in Corollary $\ref{coro strichartz estimates wave compact no boundary}$ holds true once one considers the linear fractional wave equations on $\R^d, d\geq 1$ equipped with a smooth metric $g$ satisfying $(\ref{assump elliptic}), (\ref{assump bounded metric})$. 
\end{rem}
We next give applications of the Strichartz estimates given in Theorem $\ref{theorem strichartz estimate compact no boundary}$. Let us consider the following nonlinear fractional Schr\"odinger equation
\[
\left\{
\begin{array}{ccl}
i\partial_t u(t,x) + \Lambda_g^{\sigma} u(t,x) &=& -\mu (|u|^{\nu-1} u)(t,x), \quad (t, x) \in I \times M, \mu \in \{\pm 1\}, \\
u(0,x) &=& u_0(x), \quad x\in M.
\end{array}
\right.
\tag{NLFS}
\]
with the exponent $\nu>1$. The number $\mu=1$ (resp. $\mu=-1$) corresponds to the defocusing case (resp. focusing case). By a standard approximation (see e.g. \cite{Ginibre}), the following quantities are conserved by the flow of the equations,
\begin{align}
M(u)&= \int_{M} |u(t,x)|^2 d\text{vol}_g(x), \nonumber \\
E(u)&= \int_{M} \frac{1}{2}|\Lambda_g^{\sigma/2} u(t,x)|^2 + \frac{\mu}{\nu+1}|u(t,x)|^{\nu+1} d\text{vol}_g(x). \nonumber
\end{align}
Theorem $\ref{theorem strichartz estimate compact no boundary}$ gives the following local well-posedness result.
\begin{theorem} \label{theorem local wellposedness fractional schrodinger subcritical} Consider $(M,g)$ a smooth compact boundaryless Riemannian manifold of dimension $d \geq 1$. Let $\sigma \in (1,\infty), \nu >1$ and $\gamma \geq 0$ be such that
\begin{align}
\left\{
\begin{array}{ll}
\gamma > 1/2 - 1/\max ( \nu-1, 4) & \text{when } d=1, \\
\gamma > d/2 - 1/\max ( \nu-1, 2) & \text{when } d \geq 2,
\end{array}
\right. \label{condition gamma fractional schrodinger subcritical}
\end{align}
and also, if $\nu$ is not an odd integer,
\begin{align}
\lceil\gamma\rceil \leq \nu, \label{assump smoothness schrodinger}
\end{align}
where $\lceil\gamma\rceil$ is the smallest positive integer greater than or equal to $\gamma$. Then for all $u_0 \in H^\gamma(M)$, there exist $T>0$ and a unique solution to $\emph{(NLFS)}$ satisfying
\[
u \in C([0,T],H^\gamma(M)) \cap L^p([0,T],L^\infty(M)),
\] 
for some $p> \max (\nu-1,4)$ when $d=1$ and  some $p>\max (\nu-1,2)$ when $d\geq 2$. Moreover, the time $T$ depends only on the size of the initial data, i.e. only on $\|u_0\|_{H^\gamma(M)}$. In the case $\sigma \in (0,1)$, the same result holds with $(\ref{condition gamma fractional schrodinger subcritical})$ is replaced by
\begin{align}
\left\{
\begin{array}{ll}
\gamma > 1/2 - \sigma/\max ( \nu-1, 4) & \text{when } d=1, \\
\gamma > d/2 - \sigma/\max ( \nu-1, 2) & \text{when } d \geq 2.
\end{array}
\right. \label{condition gamma fractional schrodinger subcritical 0-1}
\end{align}
\end{theorem}
We note that when $\nu$ is an odd integer, we have $F(\cdot)=-\mu|\cdot|^{\nu-1}\cdot \in C^\infty(\C, \C)$ and when $\nu$ is not an odd integer, condition $(\ref{assump smoothness schrodinger})$ implies $f \in C^{\lceil\gamma\rceil}(\C, \C)$. It allows us to use the fractional derivatives (see \cite{Kato95}, \cite{Dinh}). \newline 
\indent As a direct consequence of Theorem $\ref{theorem local wellposedness fractional schrodinger subcritical}$ and the conservation laws, we have the following global well-posedness result for the defocusing nonlinear fractional Schr\"odinger equation, i.e. $\mu=1$ in (NLFS). 
\begin{coro} \label{coro global existence of defocusing NLFS} Consider $(M,g)$ a smooth compact boundaryless Riemannian manifold of dimension $d \geq 1$. Let $\sigma \in (1/2,\infty)\backslash\{1\}$ when $d=1$, $\sigma > d-1$ when $d\geq 2$ and $\nu>1$ be such that if $\nu$ is not an odd integer, $\lceil\sigma/2\rceil\leq \nu$. Then for all $u_0 \in H^{\sigma/2}(M)$, there exists a unique global solution $u \in C(\R, H^{\sigma/2}(M)) \cap L^p_{\emph{loc}}(\R, L^\infty(M))$ to the defocusing $\emph{(NLFS)}$ for some $p>\max (\nu-1,4)$ when $d=1$ and some $p>\max (\nu-1, 2)$ when $d\geq 2$. 
\end{coro}
We finally give applications of Strichartz estimates given in Corollary $\ref{coro strichartz estimates wave compact no boundary}$ for the nonlinear fractional wave equation. Let us consider the following nonlinear fractional wave equation posed on $(M,g)$, 
\[
\left\{
\begin{array}{ccc}
\partial^2_t v(t,x) + \Lambda_g^{2\sigma} v(t,x) &=& -\mu (|v|^{\nu-1} v)(t,x), \quad (t, x) \in I \times M, \mu \in \{\pm 1\}, \\
v(0,x) =v_0(x),&&\partial_t v(0,x)= v_1(x), \quad x\in M.
\end{array}
\right.
\tag{NLFW}
\]
with $\sigma \in (0,\infty)\backslash\{1\}$ and the exponent $\nu>1$. In this case, the following energy is conserved under the flow of the equation, i.e.
\begin{align}
E(v, \partial_t v)= \int_{M}\frac{1}{2} |\partial_t v (t,x)|^2 + \frac{1}{2}|\Lambda_g^{\sigma} v(t,x)|^2 + \frac{\mu}{\nu+1}|v(t,x)|^{\nu+1} d\text{vol}_g(x). \nonumber
\end{align}
Using the Strichartz estimates given in Corollary $\ref{coro strichartz estimates wave compact no boundary}$, we have the following local well-posedness result.
\begin{theorem} \label{theorem local wellposedness subcritical wave}
Consider $(M,g)$ a smooth compact boundaryless Riemannian manifold of dimension $d \geq 1$. Let $\sigma \in (1,\infty), \nu >1$ and $\gamma \geq 0$ be as in $(\ref{condition gamma fractional schrodinger subcritical})$ and also, if $\nu$ is not an odd integer, $(\ref{assump smoothness schrodinger})$. Then for all $v_0 \in H^\gamma(M)$ and $v_1 \in H^{\gamma-\sigma}(M)$, there exist $T>0$ and a unique solution to $\emph{(NLFW)}$ satisfying 
\[
v \in C([0,T], H^\gamma(M)) \cap C^1([0,T], H^{\gamma-\sigma}(M)) \cap L^p([0,T], L^\infty(M)),
\] 
for some $p > \max (\nu-1,4)$ when $d=1$ and some $p >\max(\nu-1, 2)$ when $d\geq 2$. Moreover, the time $T$ depends only on the size of the initial data, i.e. only on $\|[v](0)\|_{H^\gamma(M)}$. In the case $\sigma \in (0,1)$, the same result holds with $(\ref{condition gamma fractional schrodinger subcritical 0-1})$ in place of $(\ref{condition gamma fractional schrodinger subcritical})$.
\end{theorem}
We organize this paper as follows. In Section 2, we prove the Strichartz estimates on $\R^d$ endowed with the smooth bounded metric $g$. In Section 3, we will give the proof of Strichartz estimates on compact manifolds $(M,g)$. We then prove the well-posedness results for the pure power-type of nonlinear fractional Schr\"odinger and wave equations on compact manifolds without boundary in Section 4.\newline  
\indent \textbf{Notation.} In this paper the constant may change from line to line and will be denoted by the same $C$. The notation $A \lesssim B$ means that there exists $C>0$ such that $A \leq C B$, and the one $A \sim B$ means that $A \lesssim B$ and $B \lesssim A$. For Banach spaces $X$ and $Y$, the notation $\|\cdot\|_{\Lc(X,Y)}$ denotes the operator norm from $X$ to $Y$ and $\|\cdot\|_{\Lc(X)}:= \|\cdot\|_{\Lc(X,X)}$.
\section{Strichartz estimates on $(\R^d,g)$} \label{section Rd}
\setcounter{equation}{0}
\subsection{Reduction of problem}
In this subsection, we give a reduction of Theorem $\ref{theorem local strichartz fractional schrodinger on manifolds}$ due to the Littlewood-Paley decomposition. To do so, we firstly recall some useful facts on pseudo-differential calculus. For $m \in \R$, we consider the symbol class $S(m)$ the space of smooth functions $a$ on $\R^{2d}$ satisfying
\[
|\partial^\alpha_x \partial^\beta_\xi a(x,\xi)| \leq C_{\alpha\beta} \scal{\xi}^{m-|\beta|},
\] 
for all $x, \xi \in \R^d$. We also need $S(-\infty):= \cap_{m\in \R} S(m)$. We define the semi-classical pseudo-differential operator with a symbol $a \in S(m)$ by
\[
Op_h(a) u (x):= (2\pi h)^{-d} \iint_{\R^{2d}} e^{ih^{-1}(x-y)\xi} a(x,\xi) u(y) dy d\xi,
\]
where $u \in \mathscr{S}(\R^{d})$. The following result gives the $\Lc(L^q(\R^d), L^r(\R^d))$-bound for pseudo-differential operators (see e.g. \cite[Proposition 2.4]{BTlocalstrichartz}). 
\begin{prop} \label{prop lq lr bounds}
Let $m >d$ and $a$ be a continuous function on $\R^{2d}$ smooth with respect to the second variable satisfying for all $\beta \in \N^d$, there exists $C_\beta>0$ such that for all $x, \xi \in \R^d$,
\[
|\partial^\beta_\xi a(x,\xi)| \leq C_\beta \scal{\xi}^{-m}.
\]
Then for all $1 \leq q \leq r \leq \infty$, there exists $C>0$ such that for all $h \in (0,1]$,
\[
\|Op_h(a)\|_{\Lc(L^q(\R^d), L^r(\R^d)} \leq C h^{-\left(\frac{d}{q}-\frac{d}{r}\right)}.
\]
\end{prop}
For a given $f \in C^\infty_0(\R)$, we can approximate $f(h^2P)$ in term of pseudo-differential operators. We have the following result (see e.g \cite[Proposition 2.5]{BTlocalstrichartz} or \cite[Proposition 2.1]{BGT}).
\begin{prop} \label{prop parametrix f}
Consider $\R^d$ equipped with a smooth metric $g$ satisfying $(\ref{assump elliptic})$ and $(\ref{assump bounded metric})$. Then for a given $f \in C^\infty_0(\R)$, there exist a sequence of symbols $q_j \in S(-\infty)$ satisfying $q_0= f \circ p$ and $\emph{supp}(q_j) \subset \emph{supp}(f \circ p)$ such that for all $N \geq 1$,
\[
f(h^2P)= \sum_{j=0}^{N-1} h^j Op_h(q_j)+ h^N R_N(h),
\]
and for all $m \geq 0$ and all $1 \leq q \leq r \leq \infty$, there exists $C>0$ such that for all $h \in (0,1]$,
\begin{align}
\|R_N(h)\|_{\Lc(L^q(\R^d), L^r(\R^d))} &\leq C h^{-\left(\frac{d}{q}-\frac{d}{r}\right)}. \nonumber \\
\|R_N(h)\|_{\Lc(H^{-m}(\R^d), H^m(\R^d))} &\leq C h^{-2m}. \nonumber
\end{align}
\end{prop}
A direct consequence of Proposition $\ref{prop lq lr bounds}$ and Proposition $\ref{prop parametrix f}$ is the following $\Lc(L^q(\R^d), L^r(\R^d))$-bound for $f(h^2P)$.
\begin{prop} \label{prop lq lr bounds for f(h^2P)}
Let $f\in C^\infty_0(\R)$. Then for all $1 \leq q \leq r \leq \infty$, there exists $C>0$ such that for all $h \in (0,1]$,
\[
\|f(h^2P)\|_{\Lc(L^q(\R^d), L^r(\R^d))} \leq C h^{-\left(\frac{d}{q}-\frac{d}{r}\right)}.
\]
\end{prop}
Next, we need the following version of the Littlewood-Paley decomposition (see e.g. \cite[Corollary 2.3]{BGT} or \cite[Proposition 2.10]{BTlocalstrichartz}). 
\begin{prop} \label{prop littlewood paley decomposition}
There exist $\varphi_0 \in C^\infty_0(\R)$ and $\varphi \in C^\infty_0(\R \backslash \{0\})$ such that
\[
\varphi_0(P)+ \sum_{h^{-1}: \emph{dya}} \varphi(h^2 P) = \emph{Id},
\]
where $h^{-1}:\emph{dya}$ means $h^{-1}=2^{k}, k \in \N \backslash \{0\}$. Moreover, for all $q \in [2,\infty)$, there exists $C>0$ such that for all $u \in \mathscr{S}(\R^d)$,
\[
\|u\|_{L^q(\R^d)} \leq C \Big(\sum_{h^{-1}:\emph{dya}} \|\varphi(h^2P)u\|^2_{L^q(\R^d)} \Big)^{1/2} + C\|u\|_{L^2(\R^d)}.
\]
\end{prop}
We end this subsection with the following reduction.
\begin{prop} \label{prop reduction of theorem local strichartz fractional schrodinger on manifolds}
Consider $\R^d, d\geq 1$ equipped with a smooth metric $g$ satisfying $(\ref{assump elliptic}), (\ref{assump bounded metric})$. Let $\sigma \in (0,\infty)\backslash \{1\}$ and $\varphi \in C^\infty_0(\R \backslash \{0\})$. If there exist $t_0>0$ small enough and $C>0$ such that for all $u_0 \in L^1(\R^d)$ and all $h \in (0,1]$,
\begin{align}
\|e^{ith^{-1}(h\Lambda_g)^\sigma} \varphi(h^2P)u_0\|_{L^\infty(\R^d)} \leq Ch^{-d}(1+|t|h^{-1})^{-d/2} \|u_0\|_{L^1(\R^d)}, \label{dispersive fractional schrodinger on manifolds}
\end{align}
for all $t\in [-t_0, t_0]$, then $\emph{Theorem}$ $\ref{theorem local strichartz fractional schrodinger on manifolds}$ holds true.
\end{prop}
The proof of Proposition $\ref{prop reduction of theorem local strichartz fractional schrodinger on manifolds}$ bases on the following version of $TT^\star$-criterion (see \cite{KeelTaoTTstar}, \cite[Theorem 10.7]{Zworski} or \cite[Proposition 4.1]{Zhang}).
\begin{theorem} \label{theorem TTstar criterion}
Let $(X, \Mc, \mu)$ be a $\sigma$-finite measured space, and $T: \R \rightarrow \Bc(L^2(X,\Mc,\mu))$ be a weakly measurable map satisfying, for some constants $C, \gamma, \delta >0$, 
\begin{align}
\|T(t)\|_{L^2(X) \rightarrow L^2(X)} &\leq C, \quad t \in \R, \label{energy estimates} \\
\|T(t)T(s)^\star\|_{L^1(X) \rightarrow L^\infty(X)} &\leq C h^{-\delta}(1+|t-s|h^{-1})^{-\tau}, \quad t, s \in \R. \label{dispersive estimates}
\end{align} 
Then for all pair $(p,q)$ satisfying
\[
p\in [2, \infty], \quad q\in [1,\infty], \quad (p,q,\delta) \ne (2, \infty, 1), \quad \frac{1}{p}\leq \tau\Big(\frac{1}{2}-\frac{1}{q}\Big),
\]
one has
\[
\|T(t) u\|_{L^p(\R, L^q(X))} \leq C h^{-\kappa} \|u\|_{L^2(X)},
\]
where $\kappa=\delta(1/2-1/q)-1/p$. 
\end{theorem}
\noindent \textit{Proof of Proposition $\ref{prop reduction of theorem local strichartz fractional schrodinger on manifolds}$.}  Using the energy estimates and dispersive estimates $(\ref{dispersive fractional schrodinger on manifolds})$, we can apply Theorem $\ref{theorem TTstar criterion}$ for $T(t)=\mathds{1}_{[-t_0,t_0]}(t)e^{ith^{-1}(h\Lambda_g)^\sigma} \varphi(h^2P), \delta=d, \tau=d/2$ and get
\begin{align}
\|e^{ith^{-1}(h\Lambda_g)^\sigma} \varphi(h^2P)u_0\|_{L^p([-t_0,t_0],L^q(\R^d))} \leq C h^{-(d/2-d/q-1/p)} \|u_0\|_{L^2(\R^d)}. \nonumber 
\end{align}
By scaling in time, we have
\begin{align}
\|e^{it\Lambda_g^\sigma} \varphi(h^2P) u_0\|_{L^p(h^{\sigma-1}[-t_0,t_0], L^q(\R^d))} &= h^{(\sigma-1)/p} \|e^{ith^{-1}(h\Lambda_g)^\sigma} \varphi(h^2P) u_0\|_{L^p([-t_0,t_0], L^q(\R^d))} \nonumber \\
&\leq C h^{-\gamma_{pq}} \|u_0\|_{L^2(\R^d)}. \label{reduce strichartz semiclassical}
\end{align}
Using the group property and the unitary property of Schr\"odinger operator $e^{it\Lambda_g^\sigma}$, we have the same estimates as in $(\ref{reduce strichartz semiclassical})$ for all intervals of size $2h^{\sigma-1}$. Indeed, for any interval $I_h$ of size $2h^{\sigma-1}$, we can write $I_h=[c-h^{\sigma-1}t_0,c+h^{\sigma-1}t_0]$ for some $c \in \R$ and
\begin{align}
\|e^{it\Lambda_g^\sigma} \varphi(h^2P) u_0\|_{L^p(I_h, L^q(\R^d))} &=\|e^{it\Lambda_g^\sigma}\varphi(h^2P)e^{ic\Lambda_g^\sigma}u_0\|_{L^p\left(h^{\sigma-1}[-t_0,t_0];L^q(\R^d)\right)} \nonumber \\
&\leq C h^{-\gamma_{pq}}\|e^{ic\Lambda_g^\sigma}u_0\|_{L^2(\R^d)} = Ch^{-\gamma_{pq}}\|u_0\|_{L^2(\R^d)}.\nonumber 
\end{align}
In the case $\sigma \in (1,\infty)$, we use a trick given in \cite{BGT}, i.e. cumulating $O(h^{1-\sigma})$ estimates on intervals of length $2h^{\sigma-1}$ to get estimates on any finite interval $I$. Precisely, by writing $I$ as a union of $N$ intervals $I_h$ of length $2h^{\sigma-1}$ with $N \lesssim h^{1-\sigma}$, we have
\begin{align}
\|e^{it\Lambda_g^\sigma} \varphi(h^2P) u_0\|_{L^p(I, L^q(\R^d))}\leq &\Big(\sum_{I_h} \int_{I_h} \|e^{it\Lambda_g^\sigma} \varphi(h^2P) u_0\|^p_{L^q(\R^d)} dt \Big)^{1/p}  \nonumber \\
\leq & C  N^{1/p}h^{-\gamma_{pq}} \|u_0\|_{L^2(\R^d)} \leq C h^{-\gamma_{pq}-(\sigma-1)/p} \|u_0\|_{L^2(\R^d)}. \label{reduce strichartz semiclassical 1}
\end{align}
In the case $\sigma \in (0,1)$, we can obviously bound the estimates over time intervals of size 1 by the ones of size $h^{\sigma-1}$ and obtain
\begin{align}
\|e^{it\Lambda_g^\sigma} \varphi(h^2P) u_0\|_{L^p(I, L^q(\R^d))} \leq Ch^{-\gamma_{pq}} \|u_0\|_{L^2(\R^d)}. \label{reduce strichartz semiclassical 2}
\end{align}
Moreover, we can replace the norm $\|u_0\|_{L^2(\R^d)}$ in the right hand side of $(\ref{reduce strichartz semiclassical 1})$ and $(\ref{reduce strichartz semiclassical 2})$ by $\|\varphi(h^2P)u_0\|_{L^2(\R^d)}$. Indeed, by choosing $\tilde{\varphi}\in C^\infty_0(\R \backslash \{0\})$ satisfying $\tilde{\varphi}=1$ near $\text{supp}(\varphi)$, we can write
\[
e^{ith^{-1}(h\Lambda_g)^\sigma} \varphi(h^2P)u_0 =e^{ith^{-1}(h\Lambda_g)^\sigma} \tilde{\varphi}(h^2P)\varphi(h^2P)u_0
\]
and apply $(\ref{reduce strichartz semiclassical 1})$ and $(\ref{reduce strichartz semiclassical 2})$ with $\tilde{\varphi}$ in place of $\varphi$. Now, by using the Littlewood-Paley decomposition given in Proposition $\ref{prop littlewood paley decomposition}$ and the Minkowski inequality, we have for all $(p,q)$ Schr\"odinger admissible,
\begin{align}
\|u\|_{L^p(I, L^q(\R^d))} \leq C \Big(\sum_{h^{-1}:\text{dya}} \|\varphi(h^2P)u\|^2_{L^p(I,L^q(\R^d))} \Big)^{1/2} + C\|u\|_{L^p(I,L^2(\R^d))}. \label{littlewood paley reduction}
\end{align}
We now apply $(\ref{littlewood paley reduction})$ for $u =e^{it\Lambda_g^\sigma} u_0$ together with $(\ref{reduce strichartz semiclassical 1})$ and get for $\sigma \in (1,\infty)$,
\[
\|e^{it\Lambda_g^\sigma}u_0\|_{L^p(I,L^q(\R^d))} \leq C \Big(\sum_{h^{-1}:\text{dya}} h^{-2(\gamma_{pq}+(\sigma-1)/p)}\|\varphi(h^2P)u_0\|^2_{L^2(\R^d)} \Big)^{1/2} + C\|u_0\|_{L^2(\R^d)}.
\]
Here the boundedness of $I$ is crucial to have a bound on the second term in the right hand side of $(\ref{littlewood paley reduction})$. The almost orthogonality and the fact that $\gamma_{pq}+(\sigma-1)/p \geq 1/p$ imply for $\sigma \in (1,\infty)$,
\[
\|e^{it\Lambda_g^\sigma}u_0\|_{L^p(I,L^q(\R^d))} \leq C \|u_0\|_{H^{\gamma_{pq}+(\sigma-1)/p}(\R^d)}.
\]
Similar results hold for $\sigma \in (0,1)$ with $\gamma_{pq}$ in place of $\gamma_{pq}+(\sigma -1)/p$ by using $(\ref{reduce strichartz semiclassical 2})$ instead of $(\ref{reduce strichartz semiclassical 1})$. This completes the proof.
\defendproof
\subsection{The WKB approximation} \label{subsection wkb approximation}
This subsection is devoted to the proof of dispersive estimates $(\ref{dispersive fractional schrodinger on manifolds})$. To do so, we will use the so called WKB approximation (see \cite{BGT}, \cite{BTlocalstrichartz}, \cite{Kapitanski} or \cite{Robert}), i.e. to approximate $e^{ith^{-1}(h\Lambda_g)^\sigma} \varphi(h^2P)$ in terms of Fourier integral operators. The following result is the main goal of this subsection. Let us denote $U_h(t):= e^{ith^{-1}(h\Lambda_g)^\sigma}$ for simplifying the presentation.
\begin{theorem}\label{theorem wkb fractional schrodinger equation}
Let $\sigma \in (0,\infty)\backslash\{1\}, \varphi \in C^\infty_0(\R \backslash \{0\})$, $J$ a small neighborhood of $\emph{supp}(\varphi)$ not containing the origin, $a \in S(-\infty)$ with $\emph{supp}(a)\subset p^{-1}(\emph{supp}(\varphi))$. Then there exist $t_0>0$ small enough, $S \in C^\infty([-t_0,t_0]\times \R^{2d})$ and a sequence of functions $a_{j}(t,\cdot,\cdot) \in S(-\infty)$ satisfying $\emph{supp}(a_j(t,\cdot,\cdot)) \subset p^{-1}(J)$ uniformly with respect to $t \in [-t_0,t_0]$ such that for all $N \geq 1$,
\[
U_h(t) Op_h(a) u_0 = J_N(t) u_0 + R_N(t)u_0,
\]
where 
\begin{align*}
J_N(t)u_0(x)&= \sum_{j=0}^{N-1} h^j J_h(S(t), a_j(t)) u_0 (x) \\
&= \sum_{j=0}^{N-1}h^j \left[(2\pi h)^{-d} \iint_{\R^{2d}} e^{ih^{-1}(S(t,x,\xi)-y\xi)}  a_{j}(t,x,\xi) u_0(y)dyd\xi\right],
\end{align*}
$J_N(0)=Op_h(a)$ and the remainder $R_N(t)$ satisfies for all $t\in [-t_0,t_0]$ and all $h\in (0,1]$,
\begin{align}
\|R_N(t)\|_{\Lc(L^2(\R^d))} \leq C h^{N-1}. \label{estimate remainder wkb}
\end{align}
Moreover, there exists a constant $C>0$ such that for all $t\in [-t_0,t_0]$ and all $h \in (0,1]$,
\begin{align}
\|J_N(t)\|_{\Lc(L^1(\R^d), L^\infty(\R^d))}  \leq  Ch^{-d}(1+|t|h^{-1})^{-d/2}. \label{dispersive estimate kwb fractional schrodinger}
\end{align}
\end{theorem}
\begin{rem} \label{rem dispersive estimate wkb}
Before entering to the proof of Theorem $\ref{theorem wkb fractional schrodinger equation}$, let us show that Theorem $\ref{theorem wkb fractional schrodinger equation}$ implies $(\ref{dispersive fractional schrodinger on manifolds})$. We firstly note that the study of dispersive estimates for $U_h(t) \varphi(h^2P)$ is reduced to the one of $U_h(t)Op_h(a)$ with $a \in S(-\infty)$ satisfying $\text{supp}(a) \subset p^{-1}(\text{supp}(\varphi))$. Indeed, by using the parametrix of $\varphi(h^2P)$ given in Proposition $\ref{prop parametrix f}$, we have for all $N \geq 1$,
\[
\varphi(h^2P) = \sum_{j=0}^{N-1} h^j Op_h(\tilde{q}_j) + h^N \tilde{R}_N(h),
\]
for some $\tilde{q}_j \in S(-\infty)$ satisfying $\text{supp}(\tilde{q}_j) \subset p^{-1}(\text{supp}(\varphi))$ and the remainder satisfies for all $m\geq 0$,
\[
\|\tilde{R}_N(h)\|_{\Lc(H^{-m}(\R^d), H^m(\R^d))} \leq C h^{-2m}.
\]
Since $U_h(t)$ is bounded  in $H^m(\R^d)$, the Sobolev embedding with $m > d/2$ implies
\[
\|U_h(t) \tilde{R}_N(h)\|_{\Lc(L^1(\R^d), L^\infty(\R^d))} \leq \|U_h(t) \tilde{R}_N(h)\|_{\Lc(H^{-m}(\R^d), H^m(\R^d))} \leq C h^{-2m}. 
\]
By choosing $N$ large enough, the remainder term is bounded in $\Lc(L^1(\R^d), L^\infty(\R^d))$ independent of $t,h$. We next show that Theorem $\ref{theorem wkb fractional schrodinger equation}$ gives dispersive estimates for $U_h(t)Op_h(a)$, i.e.
\begin{align}
\|U_h(t)Op_h(a)\|_{\Lc(L^1(\R^d), L^\infty(\R^d))}  \leq  Ch^{-d}(1+|t|h^{-1})^{-d/2}, \label{dispersive wkb remark}
\end{align}
for all $h\in (0,1]$ and all $t\in [-t_0,t_0]$. Indeed, by choosing $\tilde{\tilde{\varphi}} \in C^\infty_0(\R \backslash \{0\})$ which satisfies $\tilde{\tilde{\varphi}}=1$ near $\text{supp}(\varphi)$, we can write
\begin{align} 
U_h(t) Op_h(a)= \tilde{\tilde{\varphi}}(h^2P) U_h(t) Op_h(a) \tilde{\tilde{\varphi}}(h^2P) &+ (1-\tilde{\tilde{\varphi}})(h^2P) U_h(t) Op_h(a) \tilde{\tilde{\varphi}}(h^2P) \nonumber \\
  &+ U_h(t) Op_h(a) (1-\tilde{\tilde{\varphi}})(h^2P). \label{conclude dispersive wkb}
\end{align}
Using Theorem $\ref{theorem wkb fractional schrodinger equation}$, the first term is written as
\[
\tilde{\tilde{\varphi}}(h^2P) U_h(t) Op_h(a) \tilde{\tilde{\varphi}}(h^2P) = \tilde{\tilde{\varphi}}(h^2P) J_N(t) \tilde{\tilde{\varphi}}(h^2P) + \tilde{\tilde{\varphi}}(h^2P) R_N(t) \tilde{\tilde{\varphi}}(h^2P).
\] 
We learn from Proposition $\ref{prop parametrix f}$ and $(\ref{dispersive estimate kwb fractional schrodinger})$ that the first term in the right hand side is of size $O_{\Lc(L^1(\R^d), L^\infty(\R^d))}( h^{-d}(1+|t|h^{-1})^{-d/2})$ and the second one is of size $O_{\Lc(L^1(\R^d), L^\infty(\R^d))}( h^{N-1-d})$. For the second and the third term of $(\ref{conclude dispersive wkb})$, we compose to the left and the right hand side with $(P+1)^m$ for $m\geq 0$ and use the parametrix of $(1-\tilde{\tilde{\varphi}})(h^2P)$. By composing pseudo-differential operators with disjoint supports, we obtain terms of size $O_{\Lc(L^2(\R^d))}(h^\infty)$. The Sobolev embedding with $m > d/2$ implies that the second and the third terms are of size $O_{\Lc(L^1(\R^d), L^\infty(\R^d))}(h^\infty)$. By choosing $N$ large enough, we have $(\ref{dispersive wkb remark})$. 
\end{rem}
\noindent \textit{Proof of Theorem $\ref{theorem wkb fractional schrodinger equation}$.} Let us explain the strategy of the proof. As mentioned in the introduction, the main difficulty is that we do not have the exact form of the semi-classical fractional Laplace-Beltrami operator, namely $(h\Lambda_g)^\sigma$, in order to use the usual construction of \cite{BGT}. Fortunately, thanks to the support of the symbol $a$, we can replace $U_h(t)$ by $e^{ith^{-1} \psi(h^2P)}$ for some smooth, compactly supported function $\psi$. The interest of this replacement is that one can approximate $\psi(h^2P)$ in terms of pseudo-differential operators. We next use the action of pseudo-differential operators on Fourier integral operators and collect the powers of the semi-classical parameter $h$ to yield the Hamilton-Jacobi equation for the phase and a system of transport equations for the amplitudes. After solving these equations, we control the remainder terms and prove dispersive estimates for the main terms. The proof of this theorem is done by several steps. 
\paragraph{Step 1: Construction of the phase and amplitudes} 
Due to the support of $a$, we can replace $(h\Lambda_g)^\sigma$ by $\psi(h^2P)$ where $\psi(\lambda)=\tilde{\varphi}(\lambda)\sqrt{\lambda}^\sigma$ with $\tilde{\varphi} \in C^\infty_0(\R \backslash \{0\})$ and $\tilde{\varphi}=1$ on $J$. The interest of this replacement is that we can use Proposition $\ref{prop parametrix f}$ to write
\begin{align}
\psi(h^2P)= \sum_{k=0}^{N-1} h^k Op_h(q_k) + h^N R_N(h), \label{parametrix psi}
\end{align}
where $q_k \in S(-\infty)$ satisfying $q_0(x,\xi)= \psi\circ p(x,\xi)$, $\text{supp}(q_k) \subset p^{-1}(\text{supp}(\psi))$ and $R_N(h)$ is bounded in $L^2(\R^d)$ uniformly in $h \in (0,1]$. 
Next, using the fact 
\[
\frac{d}{dt}\Big(e^{-ith^{-1}\psi(h^2P)} J_N(t)\Big) = ih^{-1} e^{-ith^{-1}\psi(h^2P)} (hD_t -\psi(h^2P)) J_N(t),
\]
and $J_N(0)=Op_h(a)$, the Fundamental Theorem of Calculus gives
\begin{align}
e^{ith^{-1} \psi(h^2P)} Op_h(a) u_0 = J_N(t) u_0 -ih^{-1} \int_0^t e^{i(t-s) h^{-1} \psi(h^2P)} \left( hD_s-\psi(h^2P)\right)J_N(s)u_0 ds. \nonumber 
\end{align}
We want the last term to have a small contribution. To do this, we need to consider the action of $hD_t-\psi(h^2P)$ on $J_N(t)$. We first compute the action of $hD_t$ on $J_N(t)$ and have
\[
hD_t \circ J_N(t) = \sum_{l=0}^{N} h^l J_h(S(t), b_{l}(t)),
\]
where 
\begin{align}
b_{0}(t,x,\xi) &= \partial_tS(t,x,\xi) a_{0}(t,x,\xi), \nonumber \\
b_{l}(t,x,\xi) &= \partial_t S(t,x,\xi) a_{l}(t,x,\xi) + D_t a_{l-1}(t,x,\xi), \quad l=1,...,N-1, \nonumber \\
b_{N}(t,x,\xi) &= D_t a_{N-1}(t,x,\xi). \nonumber
\end{align} 
In order to study the action of $\psi(h^2P)$ on $J_N(t)$, we need the following action of a pseudo-differential operator on a Fourier integral operator (see e.g. \cite[Th\'eor\`eme IV.19]{Robert}, \cite[Theorem 2.5]{RuzSug} or \cite[Appendix]{BoucletthesisphD}).
\begin{prop} \label{prop action PDO on FIO}
Let $b \in S(-\infty)$ and $c \in S(-\infty)$ and $S \in C^\infty(\R^{2d})$ satisfy for all $\alpha, \beta \in \N^d, |\alpha+\beta| \geq 1$, there exists $C_{\alpha \beta}>0$,
\begin{align}
|\partial^\alpha_x \partial^\beta_\xi (S(x,\xi)-x\cdot \xi)| \leq C_{\alpha \beta}, \quad \forall x, \xi \in \R^d. \label{condition on phase}
\end{align}
Then
\[
Op_h(b)\circ J_h(S,c) = \sum_{j=0}^{N-1}h^j J_h(S,(b \triangleleft c)_j)+h^NJ_h(S,r_N(h)),
\]
where $(b\triangleleft c)_j$ is an universal linear combination of 
\[
\partial^\beta_\eta b(x, \nabla_xS(x,\xi)) \partial^{\beta-\alpha}_x c(x,\xi) \partial^{\alpha_1}_x  S(x,\xi)\cdots \partial^{\alpha_k}_x S(x,\xi),
\]
with $\alpha \leq \beta, \alpha_1+\cdots \alpha_k =\alpha$ and $|\alpha_l|\geq 2$ for all $l=1,...,k$ and $|\beta|=j$. The maps $(b,c) \mapsto (b\triangleleft c)_j$ and $(b,c)\mapsto r_N(h)$ are continuous from $S(-\infty) \times S(-\infty)$ to $S(-\infty)$ and $S(-\infty)$ respectively. In particular, we have
\begin{align}
(b\triangleleft c)_0(x,\xi)&= b(x,\nabla_xS(x,\xi)) c(x,\xi), \nonumber \\
i (b\triangleleft c)_1(x,\xi) &= \nabla_\eta b(x,\nabla_xS(x,\xi)) \cdot \nabla_xc(x,\xi) + \frac{1}{2}\emph{tr}\left( \nabla^2_{\eta} b (x,\partial_xS(x,\xi))\cdot\nabla^2_{x}S(x,\xi) \right) c(x,\xi). \nonumber 
\end{align}
\end{prop}
Using $(\ref{parametrix psi})$, we can apply \footnote{We will see later that the phase satisfies requirements of Proposition \ref{prop action PDO on FIO}.} Proposition $\ref{prop action PDO on FIO}$ and obtain
\begin{multline}
\psi(h^2P) \circ J_N(t) = \sum_{k=0}^{N-1} h^k Op_h(q_k) \circ \sum_{j=0}^{N-1} h^j J_h(S(t),a_{j}(t)) + h^N R_N(h) J_N(t), \\
 = \sum_{k+j+l=0}^{N} h^{k+j+l} J_h(S(t), (q_k \triangleleft a_{j}(t))_l) +h^{N+1} J_h(S(t), r_{N+1}(h,t) ) + h^N R_N(h) J_N(t). \nonumber
\end{multline}
This implies that 
\[
(hD_t-\psi(h^2P))J_N(t)= \sum_{r=0}^{N} h^r J_h(S(t),c_{r}(t)) - h^N R_N(h)J_N(t)-h^{N+1} J_h(S(t),r_{N+1}(h,t)),
\]
where
\begin{align}
c_{0}(t) &= \partial_tS(t) a_{0}(t)- q_0(x,\nabla_xS(t)) a_{0}(t), \nonumber \\
c_{r}(t) &= \partial_tS(t) a_{r}(t)- q_0(x,\nabla_xS(t)) a_{r}(t)+ D_t a_{r-1}(t) -(q_0 \triangleleft a_{r-1}(t))_1 - (q_1 \triangleleft a_{r-1}(t))_0 \nonumber \\
 &  - \sum_{k+j+l=r \atop j\leq r-2}(q_k \triangleleft a_{j}(t))_l, \quad r=1,...,N-1,\nonumber \\
c_{N}(t) &= D_ta_{N-1}(t) - (q_0\triangleleft a_{N-1}(t))_1 - (q_1\triangleleft a_{N-1}(t))_0 - \sum_{k+j+l=N \atop j\leq N-2}(q_k \triangleleft a_{j}(t))_l. \nonumber
\end{align}
The system of equations $c_{r}(t)=0$ for $r=0,...,N$ leads to the following Hamilton-Jacobi equation 
\begin{align}
\partial_tS(t) - q_0(x,\nabla_xS(t)) =0, \label{hamilton-jacobi equation}
\end{align}
with $S(0)=x\cdot \xi$, and transport equations
\begin{align}
D_t a_{0}(t) - (q_0 \triangleleft a_{0}(t))_1 - (q_1 \triangleleft a_{0}(t))_0 &= 0,  \label{transport 0} \\
D_t a_{r}(t) - (q_0 \triangleleft a_{r}(t))_1 - (q_1 \triangleleft a_{r}(t))_0 &= \sum_{k+j+l=r+1 \atop j\leq r-1}(q_k \triangleleft a_{j}(t))_l,     \label{transport r}
\end{align}
for $r=1,...,N-1$ with initial data
\begin{align} 
a_{0}(0) = a, \quad a_{r}(0) = 0, \quad r=1,...,N-1. \label{initial data transports}
\end{align}
The standard Hamilton-Jacobi equation gives the following result (see e.g. \cite[Th\'eor\`eme IV.14]{Robert} or Appendix $\ref{section hamilton jacobi}$).
\begin{prop}\label{prop solve hamilton jacobi wkb}
There exist $t_0>0$ small enough and a unique solution $S \in C^\infty([-t_0,t_0]\times \R^{2d})$ to the Hamilton-Jacobi equation 
\begin{align}
\left\{
\begin{array}{ccc}
\partial_t S(t,x,\xi) - q_0(x,\nabla_x S(t,x,\xi)) &=& 0, \\
S(0,x,\xi)&=& x\cdot \xi.
\end{array}
\right.\label{hamilton jacobi wkb}
\end{align}
Moreover, for all $\alpha,\beta\in \N^d$, there exists $C_{\alpha\beta}>0$ such that for all  $t\in [-t_0,t_0]$ and all $x,\xi \in \R^d$,
\begin{align} 
| \partial^\alpha_x \partial^\beta_\xi \left(S(t,x,\xi)-x\cdot \xi\right) | &\leq C_{\alpha\beta}|t|, \quad |\alpha+\beta| \geq 1, \label{estimate phase wkb 1} \\
| \partial^\alpha_x \partial^\beta_\xi (S(t,x,\xi)-x\cdot \xi - t q_0(x,\xi) ) | &\leq C_{\alpha\beta}|t|^2. \label{estimate phase wkb 2}
\end{align}
\end{prop}
Note that the phase given in Proposition $\ref{prop solve hamilton jacobi wkb}$ satisfies requirements of Proposition $\ref{prop action PDO on FIO}$. It remains to solve the transport equations $(\ref{transport 0}), (\ref{transport r})$. To do so, we rewrite these equations as
\begin{align*}
\partial_t a_{0}(t,x,\xi)- V(t,x,\xi)\cdot \nabla_x a_{0} (t,x,\xi)- f(t,x,\xi) a_{0}(t,x,\xi) &=0,   \\
\partial_t a_{r}(t,x,\xi)- V(t,x,\xi)\cdot \nabla_x a_{r} (t,x,\xi)- f(t,x,\xi) a_{r}(t,x,\xi) &=g_{r}(t,x,\xi), 
\end{align*}
for $r=1,...,N-1$ where 
\begin{align}
V(t,x,\xi) &= (\partial_\xi q_0) (x,\nabla_xS(t,x,\xi)), \nonumber \\
f(t,x,\xi) &= \frac{1}{2} \text{tr} \left[ \nabla^2_\xi q_0 (x,\nabla_x S(t,x,\xi)) \cdot \nabla^2_x S(t,x,\xi)  \right] + i q_1(x,\nabla_x S(t,x,\xi)), \nonumber \\
g_{r}(t,x,\xi) &= i \sum_{k+j+l =r+1 \atop j \leq r-1} (q_k\triangleleft a_{j}(t))_l. \nonumber
\end{align}
We now construct $a_r(t,x,\xi), r=0,..., N-1$ by the method of characteristics as follows. Let $Z(t,s,x,\xi)$ be the flow associated to $V(t,x,\xi)$, i.e.
\[
\partial_t Z(t,s,x,\xi) = -V(t,Z(t,s,x,\xi),\xi), \quad Z(s,s,x,\xi)= x.
\]
By the fact that $q_0 \in S(-\infty)$ and  $(\ref{estimate phase wkb 1})$ and using the same trick as in Lemma $\ref{lem estimate flow}$, we have
\begin{align}
|\partial^\alpha_x \partial^\beta_\xi ( Z(t,s,x,\xi)-x) | \leq C_{\alpha\beta} |t-s|, \label{estimate vector field}
\end{align}
for all $|t|,|s| \leq t_0$. Now, we can define iteratively 
\begin{align}
a_{0}(t,x,\xi) &= a(Z(0,t,x,\xi),\xi) \exp \left( \int_{0}^{t} f(s,Z(s,t,x,\xi),\xi) ds \right), \nonumber \\
a_{r}(t,x,\xi) &= \int_{0}^{t}g_{r}(s,Z(s,t,x,\xi), \xi) \exp\left( \int_\tau^t f(\tau, Z(\tau,t,x,\xi),\xi) d\tau \right) ds,  \nonumber
\end{align}
for $r=1,...,N-1$. These functions are respectively solutions to $(\ref{transport 0})$ and $(\ref{transport r})$ with initial data $(\ref{initial data transports})$ respectively. Since $\text{supp}(a) \subset p^{-1}(\text{supp}(\varphi))$, we see that for $t_0>0$ small enough, $(Z(t,s,p^{-1}(\text{supp}(\varphi))), \xi) \in p^{-1}(J)$ for all $|t|, |s| \leq t_0$. By extending $a_r(t,x,\xi)$ on $\R^{2d}$ by $a_r(t,x,\xi)=0$ for $(x,\xi) \notin p^{-1}(J)$, the functions $a_r$ are still smooth in $(x,\xi) \in \R^{2d}$. Using the fact that $a, q_k \in S(-\infty)$, $(\ref{estimate phase wkb 2})$ and $(\ref{estimate vector field})$, we have for $t_0>0$ small enough, $a_r(t, \cdot,\cdot)$ is a bounded set of $S(-\infty)$ and $\text{supp}(a_r(t,\cdot,\cdot)) \in p^{-1}(J)$ uniformly with respect to $t \in [-t_0,t_0]$.
\paragraph{Step 2: $L^2$-boundedness of remainder}
We will use the so called Kuranishi trick (see e.g. \cite{Robert}, \cite{Mizutani}). We firstly have
\[
R_N(t)= ih^{N-1} \int_{0}^{t} e^{i(t-s)h^{-1}\psi(h^2P)} \Big( R_N(h) J_N(s)+h J_h(S(s),r_{N+1}(h,s)) \Big)ds. 
\]
Using that $e^{i(t-s)h^{-1}\psi(h^2P)}$ is unitary in $L^2(\R^d)$ and Proposition $\ref{prop parametrix f}$ that $R_N(h)$ is bounded in $\Lc(L^2(\R^d))$ uniformly in $h\in (0,1]$, the estimate $(\ref{estimate remainder wkb})$ follows from the $L^2$-boundedness of $J_h(S(t), a(t))$ uniformly with respect to $h \in (0,1]$ and $t\in [-t_0,t_0]$ where $(a(t))_{t\in[-t_0,t_0]}$ is bounded in $S(-\infty)$. For $t\in [-t_0,t_0]$, we define a map on $\R^{3d}$ by
\[
\Lambda (t,x,y,\xi) := \int_0^1 \nabla_x S(t, y+s(x-y),\xi) ds.
\]
Using $(\ref{estimate phase wkb 1})$, there exists $t_0>0$ small enough so that for all $t\in [-t_0, t_0]$, 
\[
\|\nabla_x \nabla_\xi S(t,x,\xi)- I_{\R^d}\| \ll 1, \quad \forall x,\xi \in \R^d.
\]
This implies that 
\[
\|\nabla_\xi \Lambda(t,x,y,\xi)-I_{\R^d}| \leq \int_{0}^{1}\| \nabla_\xi \nabla_x S(t,y+s(x-y),\xi)-I_{\R^d}\| ds \ll 1, \quad \forall t\in [-t_0,t_0].
\]
Thus for all $t\in [-t_0,t_0]$ and all $x,y \in \R^d$, the map $\xi \mapsto \Lambda(t,x,y,\xi)$ is a diffeomorphism from $\R^d$ onto itself. If we denote $\xi \mapsto \Lambda^{-1}(t,x,y,\xi)$ the inverse map, then $\Lambda^{-1}(t,x,y,\xi)$ satisfies (see \cite{BoucletthesisphD}) that: for all $\alpha,\alpha',\beta \in \N^d$,  there exists $C_{\alpha\alpha'\beta}>0$ such that
\begin{align} \label{estimate inverse map L2 bound FIO}
|\partial^\alpha_x \partial^{\alpha'}_y\partial^\beta_\xi ( \Lambda^{-1}(t,x,y,\xi)-\xi)| \leq C_{\alpha\alpha'\beta}|t|,
\end{align}
for all $t\in [-t_0,t_0]$. Now, by change of variable $\xi \mapsto \Lambda^{-1}(t,x,y,\xi)$, the action $J_h(S(t),a(t))\circ J_h(S(t),a(t))^\star$ becomes (see \cite{Robert}) a semi-classical pseudo-differential operator with the amplitude
\[
a(t,x,\Lambda^{-1}(t,x,y,\xi)) \overline{a(t,y,\Lambda^{-1}(t,x,y,\xi)} |\det \partial_\xi \Lambda^{-1}(t,x,y,\xi)|.
\]
Using the fact that $(a(t))_{t\in [-t_0,t_0]}$ is bounded in $S(-\infty)$ and $(\ref{estimate inverse map L2 bound FIO})$, this amplitude and its derivatives are bounded. By the Calder\'on-Vaillancourt theorem, we have the result. 
\paragraph{Step 3: Dispersive estimates}
We prove the result for a general term, namely $J_h(S(t),a(t))$ with $(a(t))_{t\in[-t_0,t_0]}$ is bounded in $S(-\infty)$ satisfying $\text{supp}(a(t,\cdot,\cdot)) \in p^{-1}(J)$ for some small neighborhood $J$ of $\text{supp}(\varphi)$ not containing the origin uniformly with respect to $t \in [-t_0,t_0]$. The kernel of $J_h(S(t),a(t))$ reads
\[
K_h(t,x,y)=(2\pi h)^{-d}\int_{\R^d} e^{ih^{-1}(S(t,x,\xi)-y\xi)} a(t,x,\xi)d\xi.
\]
It suffices to show for all $t\in [-t_0,t_0]$ and all $h\in (0,1]$, $|K_h(t,x,y)| \leq  Ch^{-d}(1+|t|h^{-1})^{-d/2}$, for all $x,y \in \R^d$. We only consider the case $t\geq 0$, for $t\leq 0$ it is similar. Since the amplitude is compactly supported in $\xi$ and $a(t,x,\xi)$ is bounded uniformly in $t\in [-t_0,t_0]$ and $x,y \in \R^d$, we have $|K_h(t,x,y)| \leq C h^{-d}$. If $0 \leq t \leq h$ or $1+th^{-1} \leq 2$, then
\begin{align}
|K_h(t,x,y)|  \leq  Ch^{-d} \leq C h^{-d} (1+th^{-1})^{-d/2}. \nonumber 
\end{align}
We now can assume that $h\leq t \leq t_0$ and write the phase function as $ (S(t,x,\xi)-y\xi)/t$ with the parameter $\lambda = th^{-1} \geq 1$. By the choice of $\tilde{\varphi}$ (see Step 1 for $\tilde{\varphi}$), we see that on the support of the amplitude, i.e. on $p^{-1}(J)$, $q_0(x,\xi)=\sqrt{p(x,\xi)}^\sigma$. Thus we apply $(\ref{hamilton jacobi wkb})$ to write
\[
S(t,x,\xi)=x\cdot \xi + t\sqrt{p(x,\xi)}^\sigma +t^2\int_{0}^{1} (1-\theta) \partial^2_t S(\theta t,x,\xi)d\theta.
\] 
Next, using that $p(x,\xi)=\xi^t G(x) \xi= |\eta|^2$ with $\eta=\sqrt{G(x)}\xi$ or $\xi =\sqrt{g(x)} \eta$ where $g(x)=(g_{jk}(x))_{j,k=1}^d$ and $G(x)=(g(x))^{-1}=(g^{jk}(x))_{j,k=1}^d$, the kernel can be written as
\[
K_h(t,x,y)=(2\pi h)^{-d} \int_{\R^d} e^{i\lambda \Phi(t,x,y,\eta)} a(t,x,\sqrt{g(x)}\eta) |g(x)| d\eta,
\]
where
\[
\Phi(t,x,y,\eta)= \frac{\sqrt{g(x)} (x-y)\cdot \eta}{t} + |\eta|^\sigma +t\int_{0}^{1}(1-\theta) \partial^2_t S(\theta t,x, \sqrt{g(x)}\eta) d\theta.
\]
Recall that $|g(x)|:=\sqrt{\det g(x)}$. By $(\ref{assump elliptic})$, $\|\sqrt{G(x)}\|$ and $\|\sqrt{g(x)}\|$ are bounded from below and above uniformly in $x \in \R^d$. This implies that $\eta$ still belongs to a compact set of $\R^d$ away from zero. We denote this compact support by $\mathcal{K}$. The gradient of the phase is
\[
\nabla_\eta \Phi(t,x,y,\eta)= \frac{\sqrt{g(x)}(x-y)}{t} + \sigma\eta|\eta|^{\sigma-2} +t \Big(\int_{0}^{1} (1-\theta) (\nabla_\xi \partial^2_t S)(\theta t,x, \sqrt{g(x)}\eta) d\theta \Big) \sqrt{g(x)}.
\]
\indent Let us consider the case $|\sqrt{g(x)}(x-y)/{t}| \geq C$ for some constant $C$ large enough. Thanks to the Hamilton-Jacobi equation $(\ref{hamilton jacobi wkb})$ (see also $(\ref{second derivative of phase wkb})$, $(\ref{asumption H})$ and Lemma $\ref{lem existence small time}$) and the fact $\sigma \in (0,\infty)\backslash \{1\}$, we have for $t_0$ small enough,
\[
|\nabla_\eta \Phi| \geq |\sqrt{g(x)}(x-y)/{t}| -\sigma |\eta|^{\sigma-1}-O(t) \geq C_1.
\]
Hence we can apply the non stationary theorem, i.e. by integrating by parts with respect to $\eta$ together with the fact that for all $\beta\in \N^d$ satisfying $|\beta| \geq 2$, $|\partial^\beta_\eta \Phi(t,x,y,\eta)| \leq C_\beta$, we have for all $N\geq 1$,
\[
|K_h(t,x,y)| \leq C h^{-d}\lambda^{-N} \leq C h^{-d} (1+th^{-1})^{-d/2},
\]
provided $N$ is taken greater than $d/2$. \newline
\indent Thus we can assume that $|\sqrt{g(x)}(x-y)/t| \leq C$. In this case, we write
\[
\nabla^2_\eta \Phi(t,x,y,\eta)= \sigma |\eta|^{\sigma-2} \Big(I_{\R^d} + (\sigma-2)\frac{\eta\cdot \eta^t}{|\eta|^2}\Big)+O(t).
\]
Using that 
\[
\Big|\det \sigma |\eta|^{\sigma-2} \Big(I_{\R^d} + (\sigma-2)\frac{\eta\cdot \eta^t}{|\eta|^2}\Big) \Big| = \sigma^{d}|\sigma-1\|\eta|^{(\sigma-2)d} \geq C.
\]
Therefore, for $t_0>0$ small enough, the map $\eta\mapsto \nabla_\eta \Phi(t,x,y,\eta)$ from a neighborhood of $\mathcal{K}$ to its range is a local diffeomorphism. 
Moreover, for all $\beta\in \N^d$ satisfying $|\beta| \geq 1$, we have $|\partial^\beta_\eta \Phi(t,x,y,\eta)| \leq C_\beta$. The stationary phase theorem then implies that for all $t \in [h,t_0]$ and all $x,y \in \R^d$ satisfying $|\sqrt{g(x)}(x-y)/t| \leq C$,
\[
|K_h(t,x,y)| \leq C h^{-d} \lambda^{-d/2} \leq C h^{-d}(1+th^{-1})^{-d/2}.
\]
This completes the proof.
\defendproof
\section{Strichartz estimates on compact manifolds}
\setcounter{equation}{0}
In this section, we give the proof of Strichartz estimates on compact manifolds without boundary given in Theorem $\ref{theorem strichartz estimate compact no boundary}$. 
\subsection{Notations}
\paragraph{Coordinate charts and partition of unity} 
Let $M$ be a smooth compact Riemannian manifold without boundary. A coordinate chart $(U_\kappa,V_\kappa, \kappa)$ on $M$ comprises an homeomorphism $\kappa$ between an open subset $U_\kappa$ of $M$ and an open subset $V_\kappa$ of $\R^d$. Given $\phi \in C^\infty_0(U_\kappa)$(resp. $\chi \in C^\infty_0(V_\kappa)$), we define the pushforward of $\phi$ (resp. pullback of $\chi$) by $\kappa_*\phi:= \phi\circ \kappa^{-1}$ (resp. $\kappa^*\chi:=\chi\circ \kappa$). For a given finite cover of $M$, namely $M= \cup_{\kappa \in \Fc} U_\kappa$ with $\#\Fc< \infty$, there exist $\phi_\kappa \in C^\infty_0(U_\kappa), \kappa \in \Fc$ such that $1=\sum_{\kappa} \phi_\kappa(m)$ for all $m \in M$. 
\paragraph{Laplace-Beltrami operator}
For all coordinate chart $(U_\kappa,V_\kappa,\kappa)$, there exists a symmetric positive definite matrix $g_\kappa(x):=(g^\kappa_{jk}(x))_{j,k=1}^d$ with smooth and real valued coefficients on $V_\kappa$ such that the Laplace-Beltrami operator $P=-\Delta_g$ reads in $(U_\kappa, V_\kappa,\kappa)$ as
\[
P_\kappa:=-\kappa_*\Delta_g\kappa^* =-\sum_{j,k=1}^{d} |g_\kappa(x)|^{-1}\partial_j \Big( |g_\kappa(x)| g_\kappa^{jk}(x) \partial_k \Big),
\]
where $|g_\kappa(x)|=\sqrt{\det g_\kappa(x)}$ and $(g_\kappa^{j,k}(x))_{j,k=1}^d:=(g_\kappa(x))^{-1}$. The principal symbol of $P_\kappa$ is
\[
p_\kappa(x,\xi)= \sum_{j,k=1}^{d} g_\kappa^{jk}(x)\xi_j \xi_k.
\]
\subsection{Functional calculus}
In this subsection, we recall well-known facts on pseudo-differential calculus on manifolds (see e.g. \cite{BGT}).  
For a given $a \in S(m)$, we define the operator
\begin{align}
Op^\kappa_h(a):= \kappa^* Op_h(a) \kappa_*. \label{operator on manifold}
\end{align}
If nothing is specified about $a \in S(m)$, then the operator $Op^\kappa_h(a)$ maps $C^\infty_0(U_\kappa)$ to $C^\infty(U_\kappa)$. In the case $\text{supp}(a) \subset V_\kappa \times \R^d$, we have that $Op^\kappa_h(a)$ maps $C^\infty_0(U_\kappa)$ to $C^\infty_0(U_\kappa)$ hence to $C^\infty(M)$. 
We have the following result.
\begin{prop} \label{prop parametrix of P}
Let $\phi_\kappa \in C^\infty_0(U_\kappa)$ be an element of a partition of unity on $M$ and $\tilde{\phi}_\kappa, \tilde{\tilde{\phi}}_\kappa \in C^\infty_0(U_\kappa)$ be such that $\tilde{\phi}_\kappa=1$ near $\emph{supp}(\phi_\kappa)$ and $\tilde{\tilde{\phi}}_\kappa =1$ near $\emph{supp}(\tilde{\phi}_\kappa)$. Then for all $N \geq 1$, all $z \in [0,+\infty)$and all $h \in (0,1]$,
\[
(h^2P-z)^{-1}\phi_\kappa = \sum_{j=0}^{N-1} h^j \tilde{\phi}_\kappa Op^\kappa_h(q_{\kappa,j}(z)) \phi_\kappa + h^N R_N(z,h),
\]
where $q_{\kappa,j}(z) \in S(-2-j)$ is a linear combination of $a_k(p_\kappa-z)^{-1-k}$ for some symbol $a_k \in S(2k-j)$ independent of $z$ and 
\[
R_N(z,h)=- (h^2P-z)^{-1} \tilde{\tilde{\phi}}_\kappa Op^\kappa_h(r_{\kappa,N}(z,h)) \phi_\kappa,
\]
where $r_{\kappa,N}(z,h) \in S(-N)$ with seminorms growing polynomially in $1/\emph{dist}(z, \R^+)$ uniformly in $h\in (0,1]$ as long as $z$ belongs to a bounded set of $\C \backslash [0,+\infty)$.
\end{prop}
\begin{proof}
Let us set $\chi_\kappa:= \kappa_*\phi_\kappa$, similarly for $\tilde{\chi}_\kappa$ and $\tilde{\tilde{\chi}}_\kappa$ and get $\chi_\kappa, \tilde{\chi}_\kappa, \tilde{\tilde{\chi}}_\kappa \in C^\infty_0(V_\kappa)$ and $\tilde{\chi}_\kappa =1$ near $\text{supp}(\chi_\kappa)$ and $\tilde{\tilde{\chi}}_\kappa =1$ near $\text{supp}(\tilde{\chi}_\kappa)$. We firstly find an operator, still denoted by $P$, globally defined on $\R^d$ of the form
\begin{align}
P= -\sum_{j,k=1}^d g^{jk}(x) \partial_j \partial_k + \sum_{l=1}^d b_l(x) \partial_l, \label{global extending operator}
\end{align}
which coincides with $P_\kappa$ on a large relatively compact subset $V_0$ of $V_\kappa$. By ``large'', we mean that $\text{supp}(\tilde{\tilde{\chi}}_\kappa) \subset V_0$. For instance, we can take $P=\upsilon P_\kappa - (1-\upsilon)\Delta$ where $\upsilon \in C^\infty_0(V_\kappa)$ with values in $[0,1]$ satisfying $\upsilon=1$ on $V_0$. The principal symbol of $P$ is
\begin{align}
p(x,\xi)=\sum_{j,k=1}^{d}g^{jk}(x)\xi_j\xi_k,\quad \text{where } g^{jk}(x)= \upsilon(x) g^{jk}_\kappa(x) + (1-\upsilon(x))\delta_{jk}. \label{pricipal symbol of global extension P}
\end{align}
It is easy to see that $g(x)=(g^{jk}(x))$ satisfies $(\ref{assump elliptic})$ and $(\ref{assump bounded metric})$ and $b_l$ is bounded in $\R^d$ together with all of its derivatives. Using the standard elliptic parametrix for $(h^2P-z)^{-1}$ (see e.g \cite{Robert}), we have
\begin{align}
(h^2P-z) Op_h(q_{\kappa}(z,h)) = I + h^N Op_h(\tilde{r}_{\kappa,N}(z,h)), \label{localize parametrix}
\end{align}
where $q_\kappa(z,h)=\sum_{j=0}^{N-1}h^jq_{\kappa,j}(z)$ with $q_{\kappa,j}(z) \in S(-2-j)$ and $\tilde{r}_{\kappa,N}(z,h)\in S(-N)$ with seminorms growing polynomially in $\scal{z}/\text{dist}(z, \R^+)$ uniformly in $h\in (0,1]$. On the other hand, we can write
\begin{align}
(h^2P_\kappa-z) \tilde{\chi}_\kappa & Op_h(q_\kappa(z,h)) \chi_\kappa \nonumber \\
& = \tilde{\chi}_\kappa (h^2P_\kappa-z) Op_h(q_\kappa(z,h)) \chi_\kappa + [h^2P_\kappa, \tilde{\chi}_\kappa] Op_h(q_{\kappa}(z,h)) \chi_\kappa. \label{interchange trick}
\end{align}
Here $[h^2P_\kappa, \tilde{\chi}_\kappa]$ and $\chi_\kappa$ have coefficients with disjoint supports. Thanks to $(\ref{localize parametrix})$ and the composition of pseudo-differential operators with disjoint supports, we have 
\[
(h^2P_\kappa-z) \tilde{\chi}_\kappa Op_h(q_{\kappa}(z,h)) \chi_\kappa = \chi_\kappa +h^N \tilde{\tilde{\chi}}_\kappa Op_h(r_{\kappa,N}(z,h)) \chi_\kappa,
\]
with $r_{\kappa,N}(z,h)$ satisfying the required property. We then compose to the right and the left of above equality with $\kappa^*$ and $\kappa_*$ respectively and get
\[
(h^2P-z) \tilde{\phi}_\kappa Op_h^\kappa(q_{\kappa}(z,h)) \phi_\kappa = \phi_\kappa +h^N \tilde{\tilde{\phi}}_\kappa Op_h^\kappa(r_{\kappa,N}(z,h)) \phi_\kappa.
\]
This gives the result and the proof is complete.
\end{proof}
Next, we give an application of the parametrix given in Proposition $\ref{prop parametrix of P}$ and have the following result (see \cite[Proposition 2.1]{BGT} or \cite[Proposition 2.5]{BTlocalstrichartz}). 
\begin{prop} \label{prop parametrix f compact manifolds}
Let $\phi_\kappa, \tilde{\phi}_\kappa, \tilde{\tilde{\phi}}_\kappa$ be as in $\emph{Proposition } \ref{prop parametrix of P}$ and $f \in C^\infty_0(\R)$. Then for all $N \geq 1$ and all $h \in (0,1]$,
\begin{align}
f(h^2P)\phi_\kappa= \sum_{j=0}^{N-1} h^j \tilde{\phi}_\kappa Op^\kappa_h(a_{\kappa,j}) \phi_\kappa + h^N R_{\kappa,N}(h), \label{parametrix f compact manifolds}
\end{align}
where $a_{\kappa,j} \in S(-\infty)$ with $\emph{supp}(a_{\kappa,j}) \subset \emph{supp}(f \circ p_\kappa)$ for $j=0,...,N-1$. Moreover, for all $m \geq 0$, there exists $C>0$ such that for all $h\in (0,1]$,
\begin{align}
\|R_N(h)\|_{\Lc(H^{-m}(M), H^m(M))} \leq C h^{-2m}. \label{remainder parametrix f compact manifolds}
\end{align}
\end{prop}
\begin{proof}
The proof is essentially given in \cite[Proposition 2.1]{BGT}. For the reader's convenience, we recall the main steps. By using Proposition $\ref{prop parametrix of P}$ and the Helffer-Sj\"ostrand formula (see \cite{DimaSjos}), namely
\[
f(h^2P) = -\frac{1}{\pi} \int_{\C} \overline{\partial} \widetilde{f}(z) (h^2P-z)^{-1} dL(z),
\]
where $\widetilde{f}$ is an almost analytic extension of $f$, the Cauchy formula implies $(\ref{parametrix f compact manifolds})$ with 
\[
R_{\kappa,N}(h) = \frac{1}{\pi} \int_{\C} \overline{\partial} \widetilde{f}(z)(h^2P-z)^{-1} \tilde{\tilde{\phi}}_\kappa Op^\kappa_h(r_{\kappa,N}(z,h)) \phi_\kappa  dL(z).
\]
It remains to prove $(\ref{remainder parametrix f compact manifolds})$. This leads to study the action on $L^2(\R^d)$ of the map
\[
\int_{\C} \overline{\partial}\widetilde{f}(z) (P_\kappa+1)^{m/2} (h^2P_\kappa-z)^{-1} \tilde{\tilde{\chi}}_\kappa Op_h(r_{\kappa,N}(z,h)) \chi_\kappa  (P_\kappa+1)^{m/2} dL(z).
\]
Using a trick as in $(\ref{interchange trick})$, we can find a globally defined operator $P$ which coincides with $P_\kappa$ on the support of $\tilde{\tilde{\chi}}_\kappa$. We see that $\|(h^2P-z)^{-1}\|_{\Lc(L^2(\R^d))} \leq C |\im{z}|^{-1}$ and 
\[
(P+1)^{m/2} Op_h(r_{\kappa,N}(z,h)) \chi_\kappa (P+1)^{m/2} = h^{-2m} Op_h(\tilde{r}_{\kappa,N}(z,h)),
\]
where $\tilde{r}_{\kappa,N}(z,h) \in S(-N+2m)$ with seminorms growing polynomially in $1/\text{dist}(z,\R^+)$ uniformly in $h \in (0,1]$ which are harmless since $\widetilde{f}$ is compactly supported and $\overline{\partial}\widetilde{f}(z)= O(|\im{z}|^{\infty})$. By choosing $N$ such that $N-2m>d$, the result then follows from the $\Lc(L^2(\R^d))$ bound of pseudo-differential operator given in Proposition $\ref{prop lq lr bounds}$.
\end{proof}
A direct consequence of Proposition $\ref{prop parametrix f}$ using partition of unity and Proposition $\ref{prop lq lr bounds}$ is the following result. (see \cite[Corollary 2.2]{BGT} or \cite[Proposition 2.9]{BTlocalstrichartz}). 
\begin{coro}\label{coro lq lr bound pseudo}
Let $f \in C^\infty_0(\R)$. Then for all $1\leq q\leq r \leq \infty$, there exists $C>0$ such that for all $h \in (0,1]$,
\[
\|f(h^2P)\|_{\Lc(L^q(M),L^r(M))} \leq C h^{-\left(\frac{d}{q}-\frac{d}{r}\right)}.
\]
\end{coro}
The next proposition gives the Littlewood-Paley decomposition on compact manifolds without boundary (see \cite[Corollary 2.3]{BGT}) which is similar to Proposition $\ref{prop littlewood paley decomposition}$.
\begin{prop} \label{prop littlewood paley decomposition compact no boundary}
There exist $\varphi_0 \in C^\infty_0(\R)$ and $\varphi \in C^\infty_0(\R \backslash \{0\})$ 
such that for all $q \in [2,\infty)$, there exists $C>0$,
\[
\|u\|_{L^q(M)} \leq C \Big(\sum_{h^{-1}:\emph{dya}} \|\varphi(h^2P)u\|^2_{L^q(M)} \Big)^{1/2} + C\|u\|_{L^2(M)},
\]
for all $u \in C^\infty_0(M)$.
\end{prop}
\subsection{Reduction of problem}
In this subsection, we firstly show how to get Corollary $\ref{coro strichartz estimates wave compact no boundary}$ from Theorem $\ref{theorem strichartz estimate compact no boundary}$ and then give a reduction of Theorem $\ref{theorem strichartz estimate compact no boundary}$.  
\paragraph{Proof of Corollary $\ref{coro strichartz estimates wave compact no boundary}$} Since we are working on compact manifolds without boundary, it is well-known that there exists  an orthonormal basis $(e_j)_{j \in \N}$ of $L^2(M):=L^2(M,d\text{vol}_g)$ of $C^\infty$ functions on $M$ such that
\[
\Lambda_g^\sigma e_j = \lambda_j^\sigma e_j,
\]
with $0 \leq \lambda_0 \leq \lambda_1 \leq \lambda_2 \leq \cdots, \quad \lim_{j\rightarrow \infty} \lambda_j = +\infty$. For any $f$ a piecewise continuous function, the functional $f(\Lambda_g)$ is defined as 
\[
f(\Lambda_g)u:= \sum_{j \in \N} f(\lambda_j) u_j e_j.
\]
If we set $j_0: = \dim (\ker \Lambda_g^\sigma)$, then $\lambda_0 = \lambda
_1=\cdots= \lambda_{j_0-1}=0$ and $\lambda_j \geq \lambda_{j_0}>0$ for $j\geq j_0$. Here the number $j_0$ stands for the number of connected components of $M$ and the corresponding eigenfunctions $(e_j)_{j=0}^{j_0-1}$ are constant functions. We now define the projection on $\ker(\Lambda_g^\sigma)$ by
\[
\Pi_0 u := \sum_{j< j_0} u_j e_j, \quad \text{where } u_j:=\scal{e_j, u}_{L^2(M)}= \int_{M} \overline{e_j(x)} u(x)d\text{vol}_g(x).
\]
By the Duhamel formula, the equation $(\ref{linear fractional wave equation compact no boundary})$ can be written as
\[
v(t)= \cos (t\Lambda_g^\sigma) v_0 + \frac{\sin (t\Lambda_g^\sigma)}{\Lambda_g^\sigma}v_1 + \int_{0}^{t} \frac{\sin ((t-s)\Lambda_g^\sigma)}{\Lambda_g^\sigma} G(s)ds. 
\]
We remark that the only problem may happen on $\ker(\Lambda_g^\sigma)$ of $\frac{\sin (t\Lambda_g^\sigma)}{\Lambda_g^\sigma}$. But it is not the case because
\[
\Pi_0\frac{\sin (t\Lambda_g^\sigma)}{\Lambda_g^\sigma} v_1 = \sum_{j<j_0} \frac{\sin (t\lambda_j^\sigma)}{\lambda_j^\sigma} v_{1,j} e_j=\sum_{j<j_0} t\frac{\sin (t\lambda_j^\sigma)}{t\lambda_j^\sigma} v_{1,j} e_j =t\sum_{j<j_0} v_{1,j} e_j=t \Pi_0 v_1.
\]
Since $\ker(\Lambda_g^\sigma)$ is generated by constant functions, the local in time Strichartz estimates of $\Pi_0 v$, namely $\|\Pi_0 v\|_{L^p(I,L^q(M))}$ with $I$ a bounded interval, can be controlled by any Sobolev norms of data. Therefore, we only need to study the local in time Strichartz of $v$ away from $\ker(\Lambda_g^\sigma)$. Using the fact that
\[
\cos (t\Lambda_g^\sigma)= \frac{e^{it\Lambda_g^\sigma} + e^{-it\Lambda_g^\sigma}}{2}, \quad \sin (t\Lambda_g^\sigma)= \frac{e^{it\Lambda_g^\sigma} - e^{-it\Lambda_g^\sigma}}{2i},
\]
the Strichartz estimates $(\ref{homogeneous strichartz fractional wave with source})$ follow directly from the ones of $e^{\pm it\Lambda_g^\sigma}$ as in $(\ref{homogeneous strichartz fractional schrodinger with source})$. This gives Corollary $\ref{coro strichartz estimates wave compact no boundary}$. \defendproof \newline 
\indent We now prove Theorem $\ref{theorem strichartz estimate compact no boundary}$. To do so, we have the following reduction.
\begin{prop} \label{prop reduction strichartz fractional schrodinger compact no boundary}
Consider $(M,g)$ a smooth compact Riemannian manifold of dimension $d \geq 1$. Let $\sigma\in (0,\infty)\backslash\{1\}$ and $\varphi \in C^\infty_0(\R \backslash \{0\})$. If there exists $t_0>0$ small enough and $C>0$  such that for all $u_0 \in L^1(M)$ and all $h \in (0,1]$,
\begin{align}
\|e^{ith^{-1}(h\Lambda_g)^\sigma} \varphi(h^2P)u_0\|_{L^\infty(M)} \leq Ch^{-d}(1+|t|h^{-1})^{-d/2} \|u_0\|_{L^1(M)}, \label{dispersive fractional schrodinger on compact manifolds without boundary}
\end{align}
for all $t\in [-t_0, t_0]$, then $\emph{Theorem } \ref{theorem strichartz estimate compact no boundary}$ holds true.
\end{prop}
\begin{proof} 
The proof of homogeneous Strichartz estimates follows similarly to the one given in Proposition $\ref{prop reduction of theorem local strichartz fractional schrodinger on manifolds}$. We only give the proof of $(\ref{homogeneous strichartz fractional schrodinger with source})$, i.e. $\sigma \in (1,\infty)$, the one for $\sigma \in (0,1)$ is completely similar. 
The homogeneous part follows from $(\ref{homogeneous strichartz estimate compact no boundary})$. It remains to prove
\begin{align}
\Big\|\int_{0}^{t} e^{i(t-s)\Lambda_g^\sigma} F(s)ds\Big\|_{L^p(I, L^q(M))} \leq C \|F\|_{L^1(I,H^{\gamma_{pq}+(\sigma-1)/p}(M))}. \label{reduction homogeneous strichartz fractional schrodinger compact no boundary}
\end{align}
The estimate $(\ref{reduction homogeneous strichartz fractional schrodinger compact no boundary})$ follows easily from $(\ref{homogeneous strichartz estimate compact no boundary})$ and the Minkowski inequality (see \cite{BGT}, Corollary 2.10). Indeed, the left hand side reads
\begin{align*}
\Big\| \int_I \mathds{1}_{[0,t]}(s) e^{i(t-s)\Lambda_g^\sigma} F(s)ds\Big\|_{L^p(I, L^q(M))} &\leq \int_I \|\mathds{1}_{[0,t]}(s) e^{i(t-s)\Lambda_g^\sigma} F(s)\|_{L^p(I,L^q(M))} ds \nonumber \\
&\leq  \int_I \|e^{i(t-s)\Lambda_g^\sigma} F(s)\|_{L^p(I,L^q(M))} ds \\
&\leq C \int_I \|F(s)\|_{H^{\gamma_{pq}+(\sigma-1)/p}(M)} ds. \nonumber
\end{align*}
This gives $(\ref{reduction homogeneous strichartz fractional schrodinger compact no boundary})$ and the proof of Proposition $\ref{prop reduction strichartz fractional schrodinger compact no boundary}$ is complete.
\end{proof}
\subsection{Dispersive estimates}
This subsection devotes to prove the dispersive estimates $(\ref{dispersive fractional schrodinger on compact manifolds without boundary})$. 
Again thanks to the localization $\varphi$, we can replace $(h\Lambda_g)^\sigma$ by $\psi(h^2P)$ where $\psi(\lambda)= \tilde{\varphi}(\lambda)\sqrt{\lambda}^{\sigma}$ with $\tilde{\varphi}\in C^\infty_0(\R\backslash \{0\})$ such that $\tilde{\varphi}=1$ near $\text{supp}(\varphi)$. The partition of unity allows us to consider only on a local coordinates, namely $\sum_{\kappa} e^{ith^{-1}\psi(h^2P)} \varphi(h^2P)\phi_\kappa$. 
By using the same argument as in Remark $\ref{rem dispersive estimate wkb}$ and Proposition $\ref{prop parametrix f compact manifolds}$, the study of $e^{ith^{-1} \psi(h^2P)} \varphi(h^2P) \phi_\kappa$ is reduced to the one of $e^{ith^{-1}\psi(h^2P)} \tilde{\phi}_\kappa Op_h^\kappa (a_\kappa) \phi_\kappa$ with $a_\kappa \in S(-\infty)$ and $\text{supp}(a_\kappa) \subset \text{supp}(\varphi\circ p_\kappa)$. Let us set 
\[
u(t)=e^{ith^{-1}\psi(h^2P)} \tilde{\phi}_\kappa Op_h^\kappa (a_\kappa) \phi_\kappa u_0.
\]
We see that $u$ solves the following semi-classical evolution equation
\begin{align}
\left\{
\begin{array}{ccc}
(hD_t - \psi(h^2P)) u(t) &=& 0, \\
u_{\vert t=0} &=& \tilde{\phi}_\kappa Op^\kappa_h(a_\kappa)\phi_\kappa u_0. 
\end{array}
\right.
\label{evolution equation}
\end{align}
The WKB method allows us to construct an approximation of the solution to $(\ref{evolution equation})$ in finite time independent of $h$. To do so, we firstly choose $\phi_\kappa', \tilde{\phi}_\kappa', \tilde{\tilde{\phi}}_\kappa' \in C^\infty_0(U_\kappa)$ such that $\phi_\kappa' =1$ near $\text{supp}(\tilde{\tilde{\phi}}_\kappa)$ (see Proposition $\ref{prop parametrix of P}$ for $\tilde{\tilde{\phi}}_\kappa$), $\tilde{\phi}_\kappa'=1$ near $\text{supp}(\phi_\kappa')$ and $\tilde{\tilde{\phi}}_\kappa'=1$ near $\text{supp}(\tilde{\phi}_\kappa')$. Proposition $\ref{prop parametrix f compact manifolds}$ then implies 
\begin{align}
\psi(h^2P) \phi_\kappa' = \tilde{\phi}_\kappa' Op^\kappa_h(b_\kappa(h)) \phi_\kappa' + h^N  R'_{\kappa,N}(h), \label{parametrix of psi compact manifold}
\end{align}
where $b_\kappa(h)=\sum_{l=1}^{N-1} h^l b_{\kappa,l}$ with $b_{\kappa,l} \in S(-\infty)$ and $R'_{\kappa,N}(h)=O_{\Lc(L^2(M))}(1)$. By using the global extension operator defined in $(\ref{global extending operator})$, we can apply the construction of the WKB approximation given in Subsection $\ref{subsection wkb approximation}$ and find $t_0>0$ small enough, a function $S_\kappa \in C^\infty([-t_0,t_0]\times \R^{2d})$ and a sequence $a_{\kappa,j}(t,\cdot,\cdot) \in S(-\infty)$ satisfying $\text{supp}(a_{\kappa,j}(t,\cdot,\cdot)) \subset p^{-1}(J)$ (see $(\ref{pricipal symbol of global extension P})$ for the definition of $p$) for some small neighborhood $J$ of $\text{supp}(\varphi)$ not containing the origin uniformly in $t\in [-t_0,t_0]$ such that 
\begin{align}
(hD_t - Op_h(b_\kappa(h))) J_{\kappa,N}(t) = R_{\kappa,N}(t), \label{application of wkb on Rd}
\end{align}
where 
\[
J_{\kappa,N}(t) :=\sum_{j=0}^{N-1}h^j J_h(S_\kappa(t),a_{\kappa,j}(t)), \quad J_N(0)=Op_h(a_\kappa),
\]
satisfying for all $t \in  [-t_0,t_0]$ and all $(x,\xi)\in p^{-1}(J)$,
\begin{align}
|\partial^\alpha_x \partial^\beta_\xi (S_\kappa(t,x,\xi)-x\cdot\xi)| &\leq C_{\alpha\beta} |t|, \quad |\alpha+\beta|\geq 1, \label{estimates phase wkb compact manifold 1} \\
\Big|\partial^\alpha_x \partial^\beta_\xi (S_\kappa(t,x,\xi) - x\cdot\xi + t\sqrt{p(x,\xi)}^\sigma)\Big| &\leq C_{\alpha\beta} |t|^2, \label{estimates phase wkb compact manifold 2}
\end{align}
and for all $h \in (0,1]$,
\begin{align}
\|J_{\kappa,N}(t)\|_{\Lc(L^1(\R^d),L^\infty(\R^d))} &\leq C h^{-d}(1+|t|h^{-1})^{-d/2}, \label{dispersive estimates application on compact} \\
R_{\kappa,N}(t) &= O_{\Lc(L^2(\R^d))}(h^{N+1}). \label{estimate remainder on compact}
\end{align}
Next, we need the following micro-local finite propagation speed.
\begin{lem}\label{lem pseudo local}
Let $\sigma \in (0,\infty)\backslash \{1\}$, $\chi, \tilde{\chi} \in C^\infty_0(\R^d)$ such that $\tilde{\chi}=1$ near $\emph{supp}(\chi)$, $a(t) \in S(-\infty)$ with $\emph{supp}(a(t,\cdot,\cdot)) \subset p^{-1}(J)$ uniformly in $t\in [-t_0,t_0]$ and $S\in C^\infty([-t_0,t_0]\times \R^{2d})$ satisfy $(\ref{estimates phase wkb compact manifold 2})$ for all $t\in [-t_0,t_0]$ and all $(x,\xi)\in p^{-1}(J)$. Then for $t_0>0$ small enough, 
\[
J_h(S(t),a(t))\chi = \tilde{\chi} J_h(S(t),a(t))\chi + \tilde{R}(t),
\]
where $\tilde{R}(t)= O_{\Lc(L^2(\R^d))}(h^\infty)$. 
\end{lem}
\begin{proof}
The kernel of $J_h(S(t),a(t))\chi - \tilde{\chi} J_h(S(t),a(t))\chi$ is given by
\[
K_{h}(t,x,y)= (2\pi h)^{-d} \int_{\R^d} e^{ih^{-1}(S(t,x,\xi)-y\xi)} (1-\tilde{\chi})(x) a(t,x,\xi) \chi(y)d\xi.
\]
Using $(\ref{estimates phase wkb compact manifold 2})$, we can write for $t_0>0$ small enough, $t\in [-t_0,t_0]$ and $(x,\xi)\in p^{-1}(J)$,
\[
S(t,x,\xi)-y\xi = (x-y)\xi - t\sqrt{p(x,\xi)}^\sigma + O(t^2).
\]
By change of variables $\eta = \sqrt{G(x)} \xi$ or $\xi = \sqrt{g(x)}\eta$, we have
\[
K_h(t,x,y)= (2\pi h)^{-d} \int_{\R^d} e^{ih^{-1} \Phi(t,x,y,\eta)} (1-\tilde{\chi})(x) a(t,x,\sqrt{g(x)}\eta) \chi(y) \sqrt{\det g(x)} dx,
\]
where $\Phi(t,x,y,\xi)= \sqrt{g(x)}(x-y)\eta - t |\eta|^\sigma + O(t^2)$. Thanks to the support of $\chi$ and $\tilde{\chi}$, we see that $|x-y| \geq C$. This gives for $t_0>0$ small enough that
\[
|\nabla_\eta \Phi(t,x,y,\eta)| = |\sqrt{g(x)}(x-y) - t \sigma \eta |\eta|^{\sigma-2} + O(t^2)| \geq C (1+|x-y|).
\]
Here we also use the fact that $\|\sqrt{g(x)}\|$ is bounded from below and above (see $(\ref{pricipal symbol of global extension P})$). Using the fact that for all $\beta \in \N^d$ satisfying $|\beta| \geq 2$,
\[
|\partial^\beta_\eta \Phi(t,x,y,\eta)| \leq C_\beta,
\]
the non stationary phase theorem implies for all $N \geq 1$, all $t\in [-t_0,t_0]$ and all $x, y \in \R^d$,
\[
|K_h(t,x,y)| \leq C h^{N-d}(1+|x-y|)^{-N}.
\]
The Schur's Lemma gives $\tilde{R}(t)= O_{\Lc(L^2(\R^d))}(h^\infty)$. This ends the proof.
\end{proof}
\paragraph{Proof of dispersive estimates $(\ref{dispersive fractional schrodinger on compact manifolds without boundary})$} With the same spirit as in $(\ref{operator on manifold})$, let us set $J^\kappa_{N}(t) = \kappa^* J_{\kappa,N}(t) \kappa_*, R^\kappa_N(t) = \kappa^* R_{\kappa,N}(t) \kappa_*$ where $J_{\kappa,N}(t)$ and $R_{\kappa,N}(t)$ given in $(\ref{application of wkb on Rd})$. 
The Duhamel formula gives
\begin{align}
u(t)&=e^{ith^{-1} \psi(h^2P)} \tilde{\phi}_\kappa Op^\kappa_h(a_\kappa) \phi_\kappa u_0 \nonumber \\
&= \tilde{\phi}_\kappa J^\kappa_N(t) \phi_\kappa u_0 -ih^{-1} \int_0^t e^{i(t-s) h^{-1} \psi(h^2P)} ( hD_s-\psi(h^2P))\tilde{\phi}_\kappa J^\kappa_N(s)\phi_\kappa u_0 ds. \nonumber 
\end{align}
We aslo have from $(\ref{parametrix of psi compact manifold})$ that
\begin{align*}
(hD_s-\psi(h^2P)) & \tilde{\phi}_\kappa J^\kappa_N(s)\phi_\kappa \\
& = \tilde{\phi}_\kappa hD_s J^\kappa_N(s)\phi_\kappa - \tilde{\phi}_\kappa' Op^\kappa_h(b_\kappa(h)) \tilde{\phi}_\kappa J^\kappa_N(s)\phi_\kappa - h^N R_{\kappa,N}'(h) \tilde{\phi}_\kappa J^\kappa_N(s)\phi_\kappa.
\end{align*}
The micro-local finite propagation speed given in Lemma $\ref{lem pseudo local}$ and $(\ref{application of wkb on Rd})$ imply
\begin{align*}
(hD_s-\psi(h^2P))& \tilde{\phi}_\kappa J^\kappa_N(s)\phi_\kappa \\
&= \tilde{\phi}_\kappa' \kappa^*(hD_s - Op_h(b_\kappa(h)))J_N(s) \kappa_* \phi_\kappa - \tilde{R}_\kappa(s) - h^NR_{\kappa,N}'(h) \tilde{\phi}_\kappa J^\kappa_N(s)\phi_\kappa. \\
& = \tilde{\phi}_\kappa' R^\kappa_{N}(s) \phi_\kappa - \tilde{R}_\kappa(s) - h^NR_{\kappa,N}'(h) \tilde{\phi}_\kappa J^\kappa_N(s)\phi_\kappa,
\end{align*}
where $\tilde{R}_\kappa(s)= O_{\Lc(L^2(M))}(h^\infty)$. Here we also use the $L^2$-boundedness of pseudo-differential operators with symbols in $S(-\infty)$. We then get
\[
u(t)= \tilde{\phi}_\kappa J^\kappa_N(t) \phi_\kappa u_0 + \Rc^\kappa_{N}(t) u_0,
\]
where 
\[
\Rc^\kappa_{N}(t) u_0 = -ih^{-1}\int_{0}^{t} e^{i(t-s)h^{-1}\psi(h^2P)} (\tilde{\phi}_\kappa' R^\kappa_{N}(s) \phi_\kappa - \tilde{R}_\kappa(s) - h^NR_{\kappa,N}'(h) \tilde{\phi}_\kappa J^\kappa_N(s)\phi_\kappa) u_0 ds.
\]
By the same process as in Remark $\ref{rem dispersive estimate wkb}$ using $(\ref{dispersive estimates application on compact})$ and the fact $\Rc^\kappa_{N}(t) = O_{\Lc(L^2(M))}(h^{N-1})$ for all $t\in [-t_0,t_0]$, we obtain
\[
\|e^{ith^{-1} \psi(h^2P)} \varphi(h^2P) \phi_\kappa u_0\|_{L^\infty(M)} \leq C h^{-d} (1+|t|h^{-1})^{-d/2} |u_0\|_{L^1(M)},
\]
for all $t\in [-t_0,t_0]$. The dispersive estimates $(\ref{dispersive fractional schrodinger on compact manifolds without boundary})$ then follow from the above estimates and partition of unity. This completes the proof. \defendproof
\section{Nonlinear applications}
\setcounter{equation}{0}
In this section, we give the proofs of Theorem $\ref{theorem local wellposedness fractional schrodinger subcritical}$ and Corollary $\ref{coro global existence of defocusing NLFS}$ and Theorem $\ref{theorem local wellposedness subcritical wave}$.
\paragraph{Proof of Theorem $\ref{theorem local wellposedness fractional schrodinger subcritical}$} 
We only treat the case $\sigma \in (1,\infty)$ where we have Strichartz estimates with loss of derivatives. The one for $\sigma \in (0,1)$ is similar and essentially given in \cite[Theorem 1.7]{Dinh}. We follow the standard process (see e.g \cite{Ginibre} or \cite[Proposition 3.1]{BGT}) by using the fixed point argument in a suitable Banach space. We firstly choose $p> \max(\nu-1, 4)$ when $d=1$ and $p> \max(\nu-1, 2)$ when $d\geq 2$ such that $\gamma > d/2 - 1/p$ and then choose $q \geq 2$ such that
\[
\frac{2}{p} + \frac{d}{q} \leq  \frac{d}{2}.
\]
Let us consider
\[
X_T := \Big\{ u \in C(I,H^\gamma(M)) \cap L^p(I,H^\alpha_q(M)), \|u\|_{L^\infty(I,H^\gamma(M))} + \|u\|_{L^p(I,H^\gamma_q(M))} \leq N \Big\}
\]
equipped with the distance
\[
\|u-v\|_{X_T} := \|u-v\|_{L^\infty(I,L^2(M))} + \|u-v\|_{L^p(I,H^{-\gamma_{pq}-(\sigma-1)/p}_q(M))},
\]
where $I=[0,T]$, $T, N>0$ will be chosen later and $\alpha=\gamma-\gamma_{pq}-(\sigma-1)/p$. Here $H^\gamma_q(M):=(1-\Delta_g)^{-\gamma/2} L^q(M)$ is the generalized Sobolev space on $M$ and $H^\gamma(M):=H^\gamma_2(M)$. Using the persistence of regularity (see \cite[Theorem 1.25]{Cazenave}), we have $(X_T, \|\cdot\|_{X_T})$ is a complete metric space. By the Duhamel formula, it suffices to prove that the functional 
\begin{align}
\Phi_{u_0}(u)(t) = e^{it\Lambda_g^\sigma} u_0 -i\mu \int_0^t e^{i(t-s)\Lambda_g^\sigma} |u(s)|^{\nu-1} u(s)ds \nonumber
\end{align}
is a contraction on $X_T$. The Strichartz estimates $(\ref{homogeneous strichartz fractional schrodinger with source})$ imply
\[
\|\Phi_{u_0}(u)\|_{L^\infty(I,H^\gamma(M))}+\|\Phi_{u_0}(u)\|_{L^p(I,H^\alpha_q(M))} \lesssim \|u_0\|_{H^\gamma(M)}+ \|F(u)\|_{L^1(I,H^\gamma(M))},
\] 
where $F(u) = -\mu|u|^{\nu-1} u$. Using our assumption on $\nu$ (i.e. $\nu$ is an odd integer or $(\ref{assump smoothness schrodinger})$ otherwise), the fractional derivatives (see e.g. \cite[Appendix]{Kato95}) and H\"older inequality, we have
\begin{align*}
\|F(u)\|_{L^1(I,H^\gamma(M))} &\lesssim  \Big\| \|u(\cdot)\|^{\nu-1}_{L^\infty(M)} \|u(\cdot)\|_{H^\gamma(M)}\Big\|_{L^1(I)} \lesssim \|u\|^{\nu-1}_{L^{\nu-1}(I,L^\infty(M))} \|u\|_{L^\infty(I,H^\gamma(M))} \nonumber \\ 
&\lesssim  T^{1-\frac{\nu-1}{p}} \|u\|^{\nu-1}_{L^p(I,L^\infty(M))} \|u\|_{L^\infty(I,H^\gamma(M))}. 
\end{align*}
Note that by working in local coordinates, the fractional derivatives on compact manifold are reduced to the ones on $\R^d$. Similarly, using the fact that for all $z,\zeta \in \C$,
\begin{align}
|F(z)-F(\zeta)| \lesssim |z-\zeta| (|z|^{\nu-1} + |\zeta|^{\nu-1}), \label{estimate of F}
\end{align}
we have
\begin{align}
\|F(u)-F(v)\|_{L^1(I,L^2(M))} &\lesssim \Big(\|u\|^{\nu-1}_{L^{\nu-1}(I, L^\infty(M))} + \|v\|^{\nu-1}_{L^{\nu-1}(I, L^\infty(M))} \Big) \|u-v\|_{L^\infty(I,L^2(M))} \nonumber \\
&\lesssim T^{1-\frac{\nu-1}{p}} \Big(\|u\|^{\nu-1}_{L^p(I,L^\infty(M))} + \|v\|^{\nu-1}_{L^p(I,L^\infty(M))} \Big) \|u-v\|_{L^\infty(I, L^2(M))}. \nonumber
\end{align}
The Sobolev embedding with $\alpha > d/q$ implies $L^p(I,H^\alpha_q(M))\subset L^p(I,L^\infty(M))$. Thus, 
\begin{multline*}
\|\Phi_{u_0}(u)\|_{L^\infty(I,H^\gamma(M))} +\|\Phi_{u_0}(u)\|_{L^p(I,H^\alpha_q(M))} \\
\lesssim \|u_0\|_{H^\gamma(M)} + T^{1-\frac{\nu-1}{p}} \|u\|^{\nu-1}_{L^p(I,H^\alpha_q(M))} \|u\|_{L^\infty(I,H^\gamma(M))},
\end{multline*}
and
\begin{align}
\|\Phi_{u_0}(u)-\Phi_{u_0}(v)\|&_{L^\infty(I,L^2(M))}+\|\Phi_{u_0}(u)-\Phi_{u_0}(v)\|_{L^p(I,H^{-\gamma_{pq}-(\sigma-1)/p}_q(M))} \nonumber \\
&\lesssim T^{1-\frac{\nu-1}{p}} \Big(\|u\|^{\nu-1}_{L^p(I,L^\infty(M))} + \|v\|^{\nu-1}_{L^p(I,L^\infty(M))} \Big)\|u-v\|_{L^\infty(I, L^2(M))} \label{uniqueness schrodinger} \\
&\lesssim T^{1-\frac{\nu-1}{p}} \Big(\|u\|^{\nu-1}_{L^p(I,H^\alpha_q(M))} + \|v\|^{\nu-1}_{L^p(I,H^\alpha_q(M))} \Big)\|u-v\|_{L^\infty(I, L^2(M))}. \nonumber
\end{align}
This implies for all $u, v\in X_T$, there exists $C>0$ independent of $u_0\in H^\gamma(M)$ such that
\begin{align*}
\|\Phi_{u_0}(u)\|_{L^\infty(I, H^\gamma(M))}+ \|\Phi_{u_0}(u)\|_{L^p(I,H^\alpha_q(M))} \leq C\|u_0\|_{H^\gamma(M)} + CT^{1-\frac{\nu-1}{p}} N^\nu, 
\end{align*}
and 
\[
\|\Phi_{u_0}(u)-\Phi_{u_0}(v)\|_{X_T} \leq C T^{1-\frac{\nu-1}{p}} N^{\nu-1} \|u-v\|_{X_T}.
\]
Thus, if we set $N=2C\|u_0\|_{H^\gamma(M)}$ and choose $T$ small enough so that $CT^{1-\frac{\nu-1}{p}} N^{\nu-1} \leq \frac{1}{2}$, then $X_T$ is stable by $\Phi_{u_0}$ and $\Phi_{u_0}$ is a contraction on $X_T$. The fixed point theorem gives the existence of solution $u \in C(I, H^\gamma(M)) \cap L^p(I, L^\infty(M))$ to (NLFS). It remains to show the uniqueness. Consider $u, v \in C(I, H^\gamma(M))\cap L^p(I, L^\infty(M))$ two solutions of (NLFS). Since the uniqueness is a local property (see also \cite[Chapter 4]{Cazenave}), it suffices to show $u=v$ for $T$ is small. Using $(\ref{uniqueness schrodinger})$, we have
\[
\|u-v\|_{X_T} \leq CT^{1-\frac{\nu-1}{p}} \Big(\|u\|^{\nu-1}_{L^p(I, L^\infty(M))} +\|v\|^{\nu-1}_{L^p(I, L^\infty(M))} \Big)\|u-v\|_{X_T}.
\]
Since $\|u\|_{L^p(I,L^\infty(M))}$ is small if $T$ is small and similarly for $v$, we see that if $T>0$ small enough,
\[
\|u-v\|_{X_T} \leq \frac{1}{2} \|u-v\|_{X_T}.
\]
This completes the proof.\defendproof
\paragraph{Proof of Corollary $\ref{coro global existence of defocusing NLFS}$}
By the assumptions given in Corollary $\ref{coro global existence of defocusing NLFS}$, we apply Theorem $\ref{theorem local wellposedness fractional schrodinger subcritical}$ with $\gamma=\sigma/2$ and see that for all $u_0 \in H^{\sigma/2}(M)$, there exist $T>0$ and a unique solution $u \in C([0,T], H^{\sigma/2}(M)) \cap L^p([0,T], L^\infty(M))$ to the defocusing (NLFS). Note that the time $T$ depends only on $\|u_0\|_{H^{\sigma/2}(M)}$. Moreover, by a classical approximation argument, the following quantities are conserved for $u_0 \in H^{\sigma/2}(M)$,
\begin{align}
\|u(t)\|^2_{L^2(M)} &= M(u_0), \nonumber \\
\frac{1}{2}\|\Lambda_g^{\sigma/2}u(t)\|^2_{L^2(M)} + \frac{1}{\nu+1}\|u(t)\|^{\nu+1}_{L^{\nu+1}(M)} &= E(u_0). \nonumber 
\end{align}
This shows that $\|u(t)\|_{H^{\sigma/2}(M)}$ remains bounded for all $t$ in the existence domain. Thus we can apply Theorem $\ref{theorem local wellposedness fractional schrodinger subcritical}$ again with the initial data starting at $T$ and obtain a unique solution $u \in C([0,2T], H^{\sigma/2}(M)) \cap L^p([0,2T],L^\infty(M))$. By repeating this process, we extend the solution for positive times. Similarly, the same result holds for negative times. This ends the proof. \defendproof
\paragraph{Proof of Theorem $\ref{theorem local wellposedness subcritical wave}$} 
The proof is very close to the one of Theorem $\ref{theorem local wellposedness fractional schrodinger subcritical}$. We only consider the case $\sigma \in (1,\infty)$, the one for $\sigma \in (0,1)$ is similar (see also \cite[Theorem 1.13]{Dinh}). Let $(p,q)$ and $\alpha$ be as in the proof of Theorem $\ref{theorem local wellposedness fractional schrodinger subcritical}$. We will solve (NLFW) in
\begin{multline}
Y_T := \Big\{ v \in C(I,H^\gamma(M)) \cap C^1(I, H^{\gamma-\sigma}(M)) \cap L^p(I,H^\alpha_q(M)), \Big.\\
\Big. \|[v]\|_{L^\infty(I,H^\gamma(M))}+ \|v\|_{L^p(I, H^\alpha_q(M))} \leq N \Big\} \nonumber
\end{multline}
equipped with the distance
\[
\|v-w\|_{Y_T} := \|[v-w]\|_{L^\infty(I,L^2(M))} + \|v-w\|_{L^p(I,H^{- \gamma_{pq}-(\sigma-1)/p}_q(M))},
\]
where $I=[0,T]$ and $T, N>0$ will be chosen later. Here we denote
\[
\|[v]\|_{L^\infty(I,H^\gamma(M))}= \|v\|_{L^\infty(I,H^\gamma(M))} +\|\partial_t v\|_{L^\infty(I,H^{\gamma-\sigma}(M))}.
\] 
The persistence of regularity implies that $(Y_T, \|\cdot\|_{Y_T})$ is a complete metric space. By the Duhamel formula, it suffices to prove that the functional 
\begin{align}
\Phi_{v_0,v_1}(v)(t)= \cos (t\Lambda_g^\sigma) v_0 + \frac{\sin (t\Lambda_g^\sigma)}{\Lambda_g^\sigma} v_1 - \mu \int_0^t \frac{\sin((t-s)\Lambda_g^\sigma)}{\Lambda_g^\sigma}|v(s)|^{\nu-1}v(s)ds \label{duhamel formula wave}
\end{align}
is a contraction on $Y_T$. The local Strichartz estimates $(\ref{homogeneous strichartz fractional wave with source})$ imply
\begin{align*}
\|\left[\Phi_{v_0,v_1}(v)\right]\|_{L^\infty(I,H^\gamma(M))}+ \|\Phi_{v_0,v_1}(v)\|_{L^p(I, H^\alpha_q(M))} & \lesssim \|[v](0)\|_{H^\gamma(M)}+ \|F(v)\|_{L^1(I,H^{\gamma-\sigma}(M))} \\
&\lesssim \|[v](0)\|_{H^\gamma(M)}+ \|F(v)\|_{L^1(I,H^{\gamma}(M))}. 
\end{align*}
As in the proof of Theorem $\ref{theorem local wellposedness fractional schrodinger subcritical}$, the fractional derivatives with the assumption on $\nu$ given in Theorem $\ref{theorem local wellposedness subcritical wave}$, the H\"older inequality imply
\[
\|F(v)\|_{L^1(I,H^\gamma(M))} \lesssim  T^{1-\frac{\nu-1}{p}} \|v\|^{\nu-1}_{L^p(I,L^\infty(M))} \|v\|_{L^\infty(I,H^\gamma(M))}.
\]
Similarly, using $(\ref{estimate of F})$, we have
\[
\|F(v)-F(w)\|_{L^1(I,L^2(M))} \lesssim  T^{1-\frac{\nu-1}{p}} \Big(\|v\|^{\nu-1}_{L^p(I,L^\infty(M))} + \|w\|^{\nu-1}_{L^p(I,L^\infty(M))}\Big) \|u-v\|_{L^\infty(I,L^2(M))}.
\]
The Sobolev embedding $L^p(I,H^\alpha_q(M)) \subset L^p(I, L^\infty(M))$ then implies that
\begin{multline*}
\|\left[\Phi_{v_0,v_1}(v)\right]\|_{L^\infty(I,H^\gamma(M))}+ \|\Phi_{v_0,v_1}(v)\|_{L^p(I, H^\alpha_q(M))}  \\
\lesssim \|[v](0)\|_{H^\gamma(M)}+  T^{1-\frac{\nu-1}{p}} \|v\|^{\nu-1}_{L^p(I,H^\alpha_q(M))} \|v\|_{L^\infty(I,H^\gamma(M))},
\end{multline*}
and 
\begin{align}
\|\Phi_{v_0,v_1}(v)-\Phi_{v_0,v_1}(w)\|_{Y_T} & \lesssim T^{1-\frac{\nu-1}{p}} \Big(\|v\|^{\nu-1}_{L^p(I,L^\infty(M))} + \|w\|^{\nu-1}_{L^p(I,L^\infty(M))}\Big) \|u-v\|_{Y_T} \label{uniqueness wave} \\
&\lesssim T^{1-\frac{\nu-1}{p}} \Big(\|v\|^{\nu-1}_{L^p(I,H^\alpha_q(M))} + \|w\|^{\nu-1}_{L^p(I,H^\alpha_q(M))}\Big) \|u-v\|_{Y_T}. \nonumber
\end{align}
Therefore, for all $v, w \in Y_T$, there exists a constant $C>0$ independent of $v_0,v_1$ such that 
\[
\|\left[\Phi_{v_0,v_1}(v)\right]\|_{L^\infty(I,H^\gamma(M))}+ \|\Phi_{v_0,v_1}(v)\|_{L^p(I, H^\alpha_q(M))} \leq C\|[v](0)\|_{H^\gamma(M)}+  CT^{1-\frac{\nu-1}{p}} N^\nu,
\]
and 
\[
\|\Phi_{v_0,v_1}(v)-\Phi_{v_0,v_1}(w)\|_{Y_T} \leq C T^{1-\frac{\nu-1}{p}} N^{\nu-1} \|u-v\|_{Y_T}.
\]
Setting $N=2C \|[v](0)\|_{H^\gamma(M)}$ and choosing $T>0$ small enough so that $CT^{1-\frac{\nu-1}{p}} N^{\nu-1} \leq \frac{1}{2}$, we see that $Y_T$ is stable by $\Phi_{v_0,v_1}$ and $\Phi_{v_0,v_1}$ is a contraction on $Y_T$. By the fixed point theorem, there exists a unique solution $v \in Y_T$ to (NLFW). The uniqueness of solution $v \in C(I, H^\gamma(M)) \cap C^1(I, H^{\gamma-\sigma}(M)) \cap L^p(I, L^\infty(M))$ follows as in the proof of Theorem $\ref{theorem local wellposedness fractional schrodinger subcritical}$ using $(\ref{uniqueness wave})$. \defendproof
\appendix
\renewcommand*{\thesection}{\Alph{section}}
\section{Hamilton-Jacobi equation} \label{section hamilton jacobi}
\setcounter{equation}{0}
In this appendix, we will recall the standard Hamilton-Jacobi equation (see e.g. \cite[Th\'eor\`eme IV.14]{Robert}). Let us consider the following Hamilton-Jacobi equation 
\begin{align}\label{hamilton jacobi standard}
\left\{
\begin{array}{ccc}
\partial_t S(t,x,\xi) + H(x,\nabla_xS(t,x,\xi)) &=& 0, \\
S(0,x,\xi)&=& x\cdot \xi,
\end{array}
\right.
\end{align}
where $H \in C^\infty(\R^{2d})$ satisfies that for all $\alpha,\beta \in \N^d$, there exists $C_{\alpha\beta}>0$ such that for all $x,\xi \in \R^d$,
\begin{align} \label{asumption H}
|\partial^\alpha_x \partial^\beta_\xi H(x,\xi) | \leq C_{\alpha\beta}.
\end{align}
The Hamiltonian flow associated to $H$ is denoted by $\Phi_H(t,x,\xi):=(X(t,x,\xi),\Xi(t,x,\xi))$ where
\[
\left\{
\begin{array}{ccc}
\dot{X}(t) &=& \nabla_\xi H(X(t),\Xi(t)), \\
\dot{\Xi}(t) &=& -\nabla_x H(X(t),\Xi(t)),
\end{array}
\right.
\text{ and }
\left\{
\begin{array}{ccc}
X(0) &=& x, \\
\Xi(0) &=& \xi.
\end{array}
\right.
\]
Let us start with the following bound on derivatives of the Hamiltonian flow.
\begin{lem}\label{lem estimate flow}
Let $t_0 \geq 0$ and $\alpha,\beta\in \N^d$ be such that $|\alpha|+|\beta| \geq 1$. Then there exists $C_{\alpha\beta t_0} > 0$ such that for all $t\in [-t_0,t_0]$ and all $(x,\xi) \in \R^{2d}$,
\[
| \partial^\alpha_x \partial^\beta_\xi (\Phi_H(t,x,\xi)- (x,\xi) | \leq C_{\alpha\beta t_0} |t|.
\]
\end{lem}
\begin{proof}
The proof is essentially given in \cite[Lemme IV.9]{Robert}. We assume first $|\alpha+\beta|=1$ and denote
\[
Z(t)= \left( \begin{array}{cc}
\nabla_x X(t) & \nabla_\xi X(t) \\
\nabla_x \Xi(t) & \nabla_\xi \Xi(t)
\end{array}\right).
\]
By direct computation, we have
\begin{align} \label{derivative flow}
\frac{d}{dt} Z(t) = A(t) Z(t),
\end{align}
where 
\[
A(t) = \left( 
\begin{array}{cc}
\nabla_x \nabla_\xi H(X(t),\Xi(t)) & \nabla^2_\xi H(X(t),\Xi(t)) \\
-\nabla^2_x H(X(t),\Xi(t)) & -\nabla_\xi \nabla_x H(X(t),\Xi(t))
\end{array}
\right).
\]
This implies that
\[
\|Z(t)-I_{\R^{2d}}\| \leq \int_{0}^{t} \|A(s)\| \|Z(s)\| ds \leq N|t| +\int_{0}^{t} N \|Z(s)-I_{\R^{2d}}\| ds,
\]
where $N := \sup_{(t,x,\xi)\in [-t_0,t_0] \times \R^{2d}} \|A(t)\|$. Here $\|\cdot\|$ is the $\R^{2d\times 2d}$-matrix norm. Using Gronwall inequality, we have
\[
\|Z(t)-I_{\R^{2d}}\| \leq N|t| e^{Nt} \leq N e^{Nt_0} |t|.
\]
For $|\alpha+\beta| \geq 2$, we take the derivative of $(\ref{derivative flow})$ and apply again the Gronwall inequality.
\end{proof}
\begin{lem} \label{lem existence small time}
There exists $t_0>0$ small enough such that for all $t\in [-t_0,t_0]$ and all $\xi\in \R^d$, the map $x \mapsto X(t,x,\xi)$ is a diffeomorphism from $\R^d$ onto itself. Moreover, if we denote $x \mapsto Y(t,x,\xi)$ the inverse map, then for all $t\in [-t_0,t_0]$ and all $\alpha,\beta \in \N^d$ satisfying $|\alpha+\beta|\geq 1$, there exists $C_{\alpha\beta}>0$ such that for all $x,\xi \in \R^d$,
\begin{align}
| \partial^\alpha_x \partial^\beta_\xi (Y(t,x,\xi) - x) | \leq C_{\alpha\beta} |t|. \nonumber
\end{align}
\end{lem}
\begin{proof}
By Lemma $\ref{lem estimate flow}$ , there exists $t_0>0$ small enough such that
\[
\|\nabla_x X(t) - I_{\R^d}\| \leq \frac{1}{2},
\]
for all $t\in [-t_0,t_0]$. By Hadamard global inversion theorem, the map $x\mapsto X(t,x,\xi)$ is a diffeomorphism from $\R^d$ onto itself. Let $x\mapsto Y(t,x,\xi)$ be its inverse. By taking derivative $\partial^\alpha_x \partial^\beta_\xi$ with $|\alpha+\beta|=1$ of the following equality
\begin{align}
x=X(t,Y(t,x,\xi),\xi), \label{composition of X Y}
\end{align}
we have
\[
(\nabla_x X)(t,Y(t,x,\xi),\xi) \partial^\alpha_x \partial^\beta_\xi (Y(t,x,\xi)- x) = - \partial^\alpha_y\partial^\beta_\eta(X(t,y,\eta)-y)|_{(y,\eta)=(Y(t,x,\xi),\xi)}.
\]
By choosing $t_0$ small enough, we see that the matrix $(\partial_xX)(t,Y(t,x,\xi),\xi)$ is invertible and its inverse is bounded uniformly in $t\in[-t_0,t_0]$ and $x,\xi \in \R^d$. This implies that
\[
|\partial^\alpha_x \partial^\beta_\xi (Y(t,x,\xi)-x)| \leq C |\partial^\alpha_y\partial^\beta_\eta(X(t,y,\eta)-y)| \leq C_{\alpha\beta}|t|.
\]
For higher derivatives, we differentiate $(\ref{composition of X Y})$ and use an induction on $|\alpha+\beta|$. This completes the proof.
\end{proof}
Now, we are able to solve the Hamilton-Jacobi equation $(\ref{hamilton jacobi standard})$ and have the following result.
\begin{prop} \label{prop solve hamilton jacobi standard}
Let $t_0$ be as in $\emph{Lemma } \ref{lem existence small time}$. Then there exists a unique function $S \in C^\infty([-t_0,t_0] \times \R^{2d})$ such that $S$ solves the Hamilton-Jacobi equation $(\ref{hamilton jacobi standard})$. The solution $S$ is given by
\begin{align} \label{define phase standard}
S(t,x,\xi)= Y(t,x,\xi)\cdot \xi + \int_{0}^{t} (\xi \cdot \nabla_\xi H -H)\circ \Phi_H(s,Y(t,x,\xi),\xi) ds,
\end{align}
and $S$ satisfies
\begin{align} 
\nabla_\xi S(t)=Y(t),\quad \nabla_xS(t)= \Xi(t,Y(t),\xi), \quad \Phi_H(t, \nabla_\xi S(t),\xi)=(x,\nabla_xS(t)), \label{property of phase standard}
\end{align}
where $S(t):=S(t,x,\xi)$ and $Y(t):=Y(t,x,\xi)$. Moreover, for all $\alpha,\beta\in \N^d$, there exists $C_{\alpha\beta}>0$ such that for all $t \in [-t_0,t_0]$ and all $x,\xi \in \R^d$,
\begin{align}
|\partial^\alpha_x \partial^\beta_\xi \left(S(t,x,\xi)-x\cdot \xi \right)| &\leq C_{\alpha\beta}|t|,\quad |\alpha+\beta|\geq 1, \label{estimate phase wkb standard 1} \\
|\partial^\alpha_x \partial^\beta_\xi \left(S(t,x,\xi)-x\cdot \xi+tH(x,\xi)\right)| &\leq C_{\alpha\beta} |t|^2. \label{estimate phase wkb standard 2}
\end{align}
\end{prop}
\begin{proof}
It is well-known (see \cite[Th\'eor\`eme IV.14]{Robert}) that the function $S$ defined in $(\ref{define phase standard})$ is the unique solution to $(\ref{hamilton jacobi standard})$ and satisfies $(\ref{property of phase standard})$. It remains to prove $(\ref{estimate phase wkb standard 1})$ and $(\ref{estimate phase wkb standard 2})$. By $(\ref{property of phase standard})$ and the conservation of energy, we have
\[
H(x,\nabla_xS(t)) = H\circ \Phi_H(t,\nabla_\xi S(t),\xi) = H(\nabla_\xi S(t),\xi)=H(Y(t),\xi).
\]
This implies that 
\[
S(t,x,\xi)-x\cdot \xi = t\int_0^1 \partial_tS(\theta t,x,\xi)d\theta = - t\int_0^1 H(Y(\theta t,x,\xi),\xi)d\theta.
\]
Using $(\ref{asumption H})$ and Lemma $\ref{lem existence small time}$, we have $(\ref{estimate phase wkb standard 1})$. Next, we compute
\begin{align}
\partial^2_tS(t)&= -\partial_t \left[H(Y(t),\xi)\right] = -(\nabla_xH)(Y(t),\xi) \cdot \partial_t Y(t) \nonumber \\
&= -(\nabla_xH)(Y(t),\xi) \cdot  \nabla_\xi \left[ \partial_t S(t)  \right] =-(\nabla_xH)(Y(t),\xi) \cdot \nabla_\xi \left[ -H(Y(t),\xi)  \right] \nonumber 
\\
&= (\nabla_x H)^2(Y(t),\xi) \cdot \nabla_\xi Y(t)+ (\nabla_x H \cdot \nabla_\xi H)(Y(t),\xi). \label{second derivative of phase wkb}
\end{align}
The Taylor formula gives
\begin{multline}
S(t,x,\xi)= 
 x\cdot \xi - tH(x,\xi) \\ + t^2\int_{0}^{1} (1-\theta) \left[ (\nabla_x H)^2(Y(\theta t),\xi) \cdot \nabla_\xi Y(\theta t)+ (\nabla_x H \cdot \nabla_\xi H)(Y(\theta t),\xi) \right] d\theta. \nonumber
\end{multline}
Using again $(\ref{asumption H})$ and Lemma $\ref{lem existence small time}$, we have $(\ref{estimate phase wkb standard 2})$.
\end{proof}
\section*{Acknowledgments}
\addcontentsline{toc}{section}{Acknowledments}
The author would like to express his deep thanks to his wife-Uyen Cong for her encouragement and support. He also would like to thank his supervisor Prof. Jean-Marc BOUCLET for the kind guidance and constant encouragement. He also would like to thank the reviewers for their helpful comments and suggestions, which helped improve the manuscript.


{\sc Institut de Math\'ematiques de Toulouse, Universit\'e Toulouse III Paul Sabatier, 31062 Toulouse Cedex 9, France.} \\
\indent Email: \href{mailto:dinhvan.duong@math.univ-toulouse.fr}{dinhvan.duong@math.univ-toulouse.fr}

\begin{thebibliography}{99}
\addcontentsline{toc}{section}{References}

\bibitem{AGpseudo} {\sc S. Alinhac, P. G\'erard}, {\it Pseudo-differential operators and the Nash-Moser theorem}, Graduate Studies in Mathematics 82, AMS (2007).

\bibitem{BCDfourier} {\sc H. Bahouri, J-Y. Chemin, R. Danchin}, {\it Fourier analysis and non-linear partial differential equations}, A Series of Comprehensive Studies in Mathematics 343, Springer (2011).

\bibitem{BoucletthesisphD} {\sc J-M. Bouclet}, {\it Distributions spectrales pour des op\'erateurs perturb\'es}, PhD Thesis, Nantes University (2000).







\bibitem{BTlocalstrichartz} {\sc J-M. Bouclet, N. Tzvetkov}, {\it Strichartz estimates for long range perturbations}, Amer. J. Math. 129, 1565-1609 (2007).


\bibitem{Bourgain} {\sc J. Bourgain}, {\it Fourier transform restriction phenomena for certain lattice subsets and application to nonlinear evolution equations I. Schr\"odinger equations}, Geom. Funct. Anal. 3, 107-156 (1993).

\bibitem{Bourgainexponential} {\sc \name}, {\it Exponential sums and nonlinear Schr\"odinger equations}, Geom. Funct. Anal. 3, 157-178 (1993).

\bibitem{BGT} {\sc N. Burq, P. G\'erard, N. Tzvetkov}, {\it Strichartz inequalities and the nonlinear Schr\"odinger equation on compact manifolds}, Amer. J. Math. 126, 569-605 (2004).


\bibitem{Cazenave} {\sc T. Cazenave}, {\it  Semilinear Schrödinger equations}, Courant Lecture Notes in Mathematics 10, Courant Institute of Mathematical Sciences, AMS (2003).


\bibitem{ChenHolm} {\sc W. Chen, S. Holm}, {\it Physical interpretation of fractional diffusion-wave equation via lossy media obeying frequency power law}, Physics, Review, http://arxiv.org/abs/math-ph/0303040 (2003).



\bibitem{ChoHwangKwonLee} {\sc Y. Cho, G. Hwang, S. Kwon, S. Lee}, {\it Well-posedness and ill-posedness for the cubic fractional Schr\"odinger equations}, Discrete Contin. Dyn. Syst. 35, No. 7, 2863-2880 (2015). 

\bibitem{ChristKiselev} {\sc M. Christ, A. Kiselev}, {\it Maximal functions associated to filtrations}, J. Funct. Anal. 179, 409-425 (2001).



\bibitem{DemirbasErdoganTzirakis} {\sc S. Demirbas, M. B. Erdogan, N. Tzirakis} {\it Existence and Uniqueness theory for the fractional Schr\"odinger equation on the torus}, Adv. Lect. Math. 34: Some Topics in Harmonic Analysis and Applications, 145-162 (2016).

\bibitem{Dinh} {\sc V-D. Dinh}, {\it Well-posedness of nonlinear fractional Schr\"odinger and wave equations in Sobolev spaces}, preprint http://arxiv.org/abs/1609.06181.

\bibitem{DimaSjos} {\sc M. Dimassi, J. Sj\"ostrand}, {\it Spectral asymptotics in the semi-classical limit}, London Math. Soc. Lecture Note 268, Cambridge University Press (1999).

\bibitem{Ginibre} {\sc J. Ginibre}, {\it Introduction aux \'equations de Schr\"odinger non lin\'eaires}, Cours de DEA 1994-1995, Paris Onze \'edition L 161.



\bibitem{GuoWang} {\sc Z. Guo, Y. Wang}, {\it Improved Strichartz estimates for a class of dispersive equations in the radial case and their applications to nonlinear Schr\"odinger and wave equations}, J. Anal. Math. 124, No. 1, 1-38 (2014).


\bibitem{Herrmann} {\sc R. Herrmann}, {\it Fractional calculus, An introduction for physicists}, 2nd Edition, World Scientific, Gigahedron, Germany (2014).

\bibitem{HongSire} {\sc Y. Hong, Y. Sire}, {\it On fractional Schr\"odinger equations in Sobolev spaces}, Commun. Pure Appl. Anal. 14, No. 6, 2265-2282 (2015).


\bibitem{IonescuPusateri} {\sc A-D. Ionescu, F. Pusateri}, {\it Nonlinear fractional Schr\"odinger equations in one dimension}, J. Func. Anal. 266, 139-176 (2014).



\bibitem{Kapitanski} {\sc L. Kapitanski}, {\it Some generalizations of the Strichartz-Brenner inequality}, Leningrad Math. J. 1, 693-726 (1990).

\bibitem{Kato95} {\sc T. Kato}, {\it On nonlinear Schr\"odinger equations. II. $H^s$-solutions and unconditional well-posedness}, J. Anal. Math. 67, 281-306 (1995).
 

\bibitem{KeelTaoTTstar} {\sc M. Keel, T. Tao}, {\it Endpoint Strichartz estimates}, Amer. J. Math. 120, No. 5, 955-980 (1998).

\bibitem{Laskin2000} {\sc N. Laskin}, {\it Fractional quantum mechanics and L\'evy path integrals}, Phys. Lett A 268, 298-305 (2000).

\bibitem{Laskin2002} {\sc \name}, {\it Fractional Schr\"odinger equation}, Phys. Rev. E 66, 056108 (2002).
  


\bibitem{Mizutani} {\sc H. Mizutani}, {\it Strichartz estimates for Schr\"odinger equations with variable coefficients and potentials at most linear at spatial infinity}, J. Math. Soc. Japan. 65, No. 3, 687-721 (2013).




\bibitem{Pausaderradial} {\sc B. Pausader}, {\it Global well-posedness for energy critical fourth-order Schr\"odinger equations in the radial case}, Dynamics of PDE 4, No.3, 197-225 (2007).

\bibitem{Pausadercubic} {\sc \name}, {\it The cubic fourth-order Schr\"odinger equation}, J. Funct. Anal. 256, 2473-2517 (2009).


\bibitem{Robert} {\sc D. Robert}, {\it Autour de l'approximation semi-classique}, Progress in mathematics 68, Birkha\"user (1987).



\bibitem{RuzSug} {\sc M. Ruzhansky, M. Sugimoto}, {\it Global calculus of Fourier integral operators, weighted estimates, and applications to global analysis of hyperbolic equations}, Operator Theory: Advances and Applications 164, 65-78 (2006).



\bibitem{Soggeoscillatory} {\sc C-D. Sogge}, {\it Oscillatory integrals and spherical harmonics}, Duke Math J. 53, 43-65 (1986).

\bibitem{Soggewave} {\sc \name}, {\it Lectures on nonlinear wave equations}, International Press (2008).




\bibitem{Stein70} {\sc E. Stein}, {\it Singular integrals and differentiability properties of functions}, Princeton Math. Series 30, Princeton University Press (1970).

\bibitem{Tao} {\sc T. Tao}, {\it Nonlinear dispersive equations: local and global analysis}, CBMS Regional Conference Series in Mathematics 106, AMS (2006).




\bibitem{Zworski} {\sc M. Zworski}, {\it Semiclassical Analysis}, Graduate Studies in Mathematics 138, AMS (2012).






\end{thebibliography}
\end{document}